\documentclass[oneside,english,british,reqno]{amsart}
\usepackage{amssymb}
\usepackage[all]{xy}
\usepackage{graphicx}

\usepackage{hyperref}

\hypersetup{linkcolor=red,citecolor=blue,filecolor=dullmagenta,urlcolor=darkblue} 

\newcommand{\ignore}[1]{}

\newcommand{\cum}{\operatorname{cum}}

\newcommand{\bb}{\mathbb}

\newcommand{\C}{\bb C} 

\newcommand{\Z}{\bb Z}

\newcommand{\R}{\bb R}
\newcommand{\N}{\bb N}

\newcommand{\vol}{\operatorname{vol}}

\newtheorem{Theorem}{Theorem}
\newtheorem{Cor}[Theorem]{Corollary}

\newtheorem{Prop}[Theorem]{Proposition}
\newtheorem{theo}[Theorem]{Theorem$^*$}
\newtheorem{Lemma}[Theorem]{Lemma}
\newtheorem*{lemma*}{Lemma}
\newtheorem*{convention*}{Convention}

\newtheorem*{theorem*}{Theorem}

\numberwithin{equation}{section}
\numberwithin{Theorem}{section}

\begin{document}
\title[Higher-order correlations]{Higher-order correlations for group actions}
\author{Alexander Gorodnik} 


\begin{abstract}
This survey paper discusses behaviour
of higher-order correlations for one-parameter dynamical systems 
and more generally for dynamical systems arising from group actions.
In particular, we present a self-contained proof of 
quantitative bounds for higher-order correlations of actions
of simple Lie groups. We also outline several applications
of our analysis of correlations that include asymptotic formulas for counting
lattice points, existence of approximate configurations in lattice subgroups, 
and validity of the Central Limit Theorem for multi-parameter group actions.
\end{abstract}

\maketitle

\setcounter{tocdepth}{1}
{\small \tableofcontents}
 
\section{Introduction}\label{sec:intro}

In this survey, we will be interested in gaining an insight into
asymptotic properties of chaotic group actions. There are several quite distinct points of view
on how this problem may be studied. Our approach here
is based on the analysis of higher-order correlations which characterise 
random-like behaviour of observables computed along orbits. For instance, let us consider a measure-preserving 
transformation $T:X\to X$ of a probability space $(X,\mu)$. Then 
for functions $\phi_1,\ldots,\phi_r\in L^\infty(X)$,
the {\it correlations of order r} are defined as 
\begin{equation}
\label{eq:corr0}
\int_X \phi_1(T^{n_1}x)\cdots \phi_r(T^{n_r}x)\, d\mu(x),\quad 
n_1,\ldots, n_r \in \mathbb{N}.
\end{equation}
The transformation $T$ is called {\it mixing of order $r$} if
for all $\phi_1,\ldots,\phi_r\in L^\infty(X)$, 
\begin{equation}
\label{eq:mmix0}
\int_X \phi_1(T^{n_1}x)\cdots \phi_r(T^{n_r}x)\, d\mu(x)\longrightarrow 
\left(\int_X\phi_1\,d\mu\right)\cdots \left(\int_X\phi_r\,d\mu\right)
\end{equation}
as $|n_i-n_j|\to \infty$ for all $i\ne j$. 
The multiple mixing property, in particular, implies that the
family of functions $\{\phi\circ T^n\}$ 
is quasi-independent asymptotically.
The study of this property was initiated by Rokhlin \cite{roh}
who showed that ergodic endomorphisms of compact abelian groups
are mixing of all orders.
In this work, Rokhlin also raised the question,
which still remains open,
whether mixing of order two implies mixing of all orders
for general measure-preserving transformations.
Kalikow \cite{kal} established this for rank-one transformations, 
Ryzhikov \cite{ryz} --- for transformations of finite rank,
and Host \cite{host} --- for transformations with singular spectrum.

The multiple mixing property has been extensively studied for 
flows on homogeneous spaces. Ornstein and Weiss \cite{ow} established that
the geodesic flow on compact hyperbolic surfaces is Bernoulli,
which implies that it is mixing of all orders.
Dani \cite{dani1,dani2} proved that a quite general partially hyperbolic one-parameter homogeneous flow satisfies the Kolmogorov property so that, in particular, it is mixing of all orders.
Sinai \cite{sinai} conjectured that the horocycle flow is also mixing of all orders. 
This conjecture was proved by Marcus \cite{marcus}.
In fact, Marcus' work established mixing of all order for general flows
on homogeneous spaces of semisimple groups. 
Ultimately Starkov \cite{starkov}, building on the work of Mozes \cite{mozes}
and the theory of unipotent flows,
proved mixing of all orders for general mixing one-parameter flows on
finite-volume homogeneous spaces.

Quantitative estimates on higher-order correlations \eqref{eq:corr0}
are also of great importance. Using Fourier-analytic techniques,
Lind \cite{lind}   proved exponential convergence of correlations 
of order two for ergodic toral automorphisms,
and P\`ene \cite{pe} proved this for correlations of all orders.
Dolgopyat \cite{dolg} established a general result about exponential convergence
of correlations for partially hyperbolic dynamical systems 
under the assumption of quantitative equidistribution of translates
of unstable manifolds. Gorodnik and Spatzier \cite{GS0}
showed exponential convergence of correlations of all orders for ergodic
automorphisms of nilmanifolds.

\medskip

More generally, we consider a measure-preserving action of a locally compact
group $G$ on a probability space space $(X,\mu)$.
For functions $\phi_1,\ldots,\phi_r\in L^\infty(X)$,
we define the {\it correlations of order $r$} as 
\begin{equation}
\label{eq:corr1}
\int_X \phi_1(g_1^{-1}x)\cdots \phi_r(g_r^{-1}x)\, d\mu(x),\quad 
g_1,\ldots, g_r \in G.
\end{equation}
We assume that the group $G$ is equipped with a (proper) metric $d$.
We say that the action is {\it mixing of order $r$} if 
for functions $\phi_1,\ldots,\phi_r\in L^\infty(X)$,
\begin{equation}
\label{eq:mmix}
\int_X \phi_1(g_1^{-1}x)\cdots \phi_r(g_r^{-1}x)\, d\mu(x)
\longrightarrow \left(\int_X\phi_1\,d\mu\right)\cdots \left(\int_X\phi_r\,d\mu\right)
\end{equation}
as $d(g_i,g_j)\to \infty$ for all $i\ne j$.
At present, the available results about higher-order mixing for
multi-parameter actions are limited to several particular classes of dynamical systems.
del Junco and Yassawi \cite{jy1,jy2} proved 
that for finite rank actions of countable abelian groups
satisfying additional technical conditions, mixing of order two implies mixing of all orders.
It was discovered by Ledrappier \cite{led} that mixing of order two
does not imply mixing of order three in general for $\mathbb{Z}^2$-actions.
The example constructed in \cite{led} is an action
by automorphisms on a (disconnected) compact abelian group.
On the other hand,
Schmidt and Ward \cite{SW} established that 
$\Z^k$-actions  by automorphisms on compact connected abelian groups
that are  mixing of order two are always mixing of all orders.
We refer to the monograph of Schmidt \cite{sch} for extensive 
study of mixing properties for higher-rank abelian actions 
by automorphisms of compact abelian groups.
It turns out that this problem is intimately connected
to deep number-theoretic questions that involve analysing solutions of $S$-unit equations.
Gorodnik and Spatzier \cite{GS} proved that mixing
$\Z^k$-actions by automorphisms on nilmanifolds are mixing of all orders.
Using Diophantine estimates on logarithms of algebraic numbers,
the work \cite{GS} also established quantitative estimates for correlations of order up to three. The problem of producing explicit quantitative bounds on general higher-order
correlations in this setting is still open, even for $\Z^k$-actions by toral automorphisms.

\medskip

In this paper, we provide a self-contained accessible treatment of the analysis 
of higher-order correlations for measure-preserving
actions of a (noncompact) simple connected Lie group $G$
with finite centre (a less advanced reader may think about the groups like $\hbox{SL}_d(\R)$ or $\hbox{Sp}_{2n}(\R)$).
This topic has long history going back at least to the works of Harish-Chandra (see, for instance, \cite{hc}). 
Indeed, the correlations of order two can be interpreted as matrix coefficients
of the corresponding unitary representation on the space $L^2(X)$,
and quantitative estimates on matrix coefficients have been established
in the framework of the Representation Theory
(see \S\ref{sec:decay} and \S\ref{sec:cor_2} for an extensive discussion).
 This in particular leads to a surprising corollary that every ergodic
 action of $G$ is always mixing. Moreover, Mozes \cite{mozes}
 proved that every mixing action of $G$ is always mixing of all order,
 and Konstantoulas \cite{konst} and Bj\"orklund, Einsiedler, Gorodnik \cite{beg}
 established quantitative estimates on higher-order correlations.
The main goal of these notes is to outline  a proof of  the following result.
Let $L$ be a connected Lie group and $\Gamma$ a discrete subgroup of $L$ with finite
 covolume. We denote by $(X,\mu)$ the space $L/\Gamma$ equipped with the 
 invariant probability measure $\mu$. 
Let $G$ be a (noncompact) simple connected higher-rank Lie group with finite center
equipped with a left-invaraint Riemannian metric $d$.
We consider the measure-preserving action of $G$ on $(X,\mu)$ given
by a smooth representation $G\to L$.
In this setting, we establish quantitative estimates on higher-order correlations:

\begin{theorem*} 
Assuming that the action of $G$ on $(X,\mu)$ is ergodic,
there exists $\delta_r>0$ (depending only on $G$ and $r$) such that 
for all elements $g_1,\ldots,g_r\in G$ and all
functions $\phi_1,\ldots,\phi_r\in C_c^\infty(X)$,
\begin{align}
\int_X \phi_1(g_1^{-1}x)\cdots \phi_r(g_r^{-1}x)\, d\mu(x)\label{eq:main}
=\,& \left(\int_X\phi_1\,d\mu\right)\cdots \left(\int_X\phi_r\,d\mu\right)\\
&\quad+O_{\phi_1,\ldots,\phi_r,r}\left( e^{-\delta_r D(g_1,\ldots,g_r)}\right),\nonumber
\end{align}
where 
$$
D(g_1,\ldots,g_r)=\min_{i\ne j} d(g_i,g_j).
$$
\end{theorem*}

As we shall explain below, a version of this theorem also holds for rank-one groups $G$
provided that the action of $G$ on $L^2(X)$ satisfies the necessary condition of having
the spectral gap. In this case, the exponent $\delta_r$ also depends on the action.

\medskip

It turns out that analysis of correlations has several far-reaching applications,
and here we outline how to use this approach 
\begin{itemize}
\item to establish an asymptotic formula for the number of lattice points,
\item to show existence of approximate configurations in lattice subgroups,
\item to prove the Central Limit Theorem for group actions.
\end{itemize}
Other interesting applications of quantitative bounds on correlations, which we do not discuss
 here, involve the Kazhdan property (T) \cite[\S V.4.1]{ht},
 the cohomological equation \cite{KS,GS0,GS}, the global rigidity of actions \cite{fks},
 and analysis of the distribution of arithmetic counting functions \cite{bg1,bg2}.

\medskip

This paper is based on a series of lectures given at the Tata Institue of Fundamental Research
which involved participants with quite diverse backgrounds ranging from starting PhD
students to senior researchers. When I was choosing the material, I was aiming 
to make it accessible, but at the same time to give a reasonably detailed
exposition of the developed methods as well as to survey current state of the art in the field.
This inevitably required some compromises. In particular, we assumed very little knowledge 
of the Theory of Lie groups, and some of the arguments are carried out only in the 
case when $G=\hbox{SL}_d(\R)$. I hope that a less prepared reader should be able to follow
this paper by thinking that a ``connected simple Lie group with finite center'' is $\hbox{SL}_d(\R)$,
and advanced readers might be able to infer  from 
our exposition how to deal with the general case.
Besides giving a self-contained proof of the bound \eqref{eq:main},
we also state a number of  more advanced results without proofs, which are indicated by 
the symbol $(^*)$.

\subsection*{Organisation of the paper}

We are not aware of any direct way for proving the main bound \eqref{eq:main},
and our arguments proceeds in several distinct steps. 
First, in \S\ref{sec:decay} and \S\ref{sec:cor_2},
we study the behaviour of correlations of order two using representation-theoretic techniques.
In \S\ref{sec:decay} we show that the correlations of order two decay at infinity
(see Theorem \ref{th:hm}), and in \S\ref{sec:cor_2} we establish quantitative bounds
on the correlations of order two (see Theorem \ref{th:high-rank2}).
Then the main bound \eqref{eq:main} is established using an elaborate inductive argument in \S\ref{sec:cor_gen}
(see Theorem \ref{th:cor_high}).
We also discuss several application of the established bounds for correlations:
in \S\ref{sec:counting} we derive an asymptotic formula for the number of lattice points,
in \S\ref{sec:conf} we establish existence of approximate configurations,
and in \S\ref{sec:clt} we prove the  Central Limit Theorem for group actions.

{\small
$$
\xymatrixcolsep{1.4pc} \xymatrixrowsep{1.4pc}
\xymatrix{
	\fbox{\begin{tabular}{c} \hbox{\bf Mixing}\\ \hbox{(\S \ref{sec:decay})}\end{tabular}} \ar[d] \ar[r] & \fbox{\begin{tabular}{c} \hbox{\bf Counting lattice points}\\ \hbox{(\S \ref{sec:counting})}\end{tabular}}\\
	\fbox{\begin{tabular}{c} \hbox{\bf Quantitative mixing}\\ \hbox{(\S \ref{sec:cor_2})}\end{tabular}} \ar[d] \ar@{-->}[ru] &  \fbox{\begin{tabular}{c} \hbox{\bf Configurations}\\ \hbox{(\S \ref{sec:conf})}\end{tabular}}\\
	\fbox{\begin{tabular}{c} \hbox{\bf Higher-order quantitative mixing}\\ \hbox{(\S \ref{sec:cor_gen})}\end{tabular}} \ar[ru]  \ar[r] & \fbox{\begin{tabular}{c} \hbox{\bf Central limit theorem}\\ \hbox{(\S \ref{sec:clt})}\end{tabular}} 
	}
$$
}

\subsection*{Acknowledgement}
This survey paper has grown out of the lecture series given by the author
at the Tata Institute of Fundamental Research in Spring 2017.
I would like to express my deepest gratitude to 
the Tata Institute for the hospitality
and to the organisers of this programme
-- Shrikrishna Dani and Anish Ghosh -- for all their hard work on setting up this event
and making it run smoothly.

\section{Decay of matrix coefficients}\label{sec:decay}

Let $G$ be a (noncompact) connected simple Lie group with finite center
(e.g., $G=\hbox{SL}_d(\R)$). We consider a measure-preserving action 
$G$ on a standard probability space $(X,\mu)$.
The goal of this section is to show 
a surprising result that ergodicity of any such action implies that it is mixing:

\begin{Theorem}\label{th:hm1}
Let $G$ be a (noncompact) connected simple Lie group with finite centre
and $G\times X\to X$ a measurable measure-preserving
action on a standard probability space $(X,\mu)$.
We assume that the action of $G$ on $(X,\mu)$ is ergodic (that is, the space $L^2(X)$ has no nonconstant $G$-invariant functions). 
Then for all $\phi,\psi\in L^2(X)$,
$$
\int_X \phi(g^{-1}x)\psi(x)\, d\mu(x)\longrightarrow \left(\int_X \phi\, d\mu\right) 
\left(\int_X \psi\, d\mu\right)
$$
as $g\to \infty$ in $G$.
\end{Theorem}

We observe that a measure-preserving action as above defines a unitary representation $\pi$ of $G$
on the space $\mathcal{H}=L^2(X)$ given by
\begin{equation}
\label{eq:rep}
\pi(g)\phi(x)=\phi(g^{-1}x)\quad \hbox{for $g\in G$ and $x\in X.$}
\end{equation}
One can also check (see, for instance, \cite[A.6]{bhv}) that 
this representation is strongly continuous (that is, the map $g\mapsto \pi(g)\phi$, $g\in G$,
is continuous).

\begin{convention*}
{\rm 	
Throughout these notes, we always implicitly assume that
representations are strongly continuous and Hilbert spaces are separable.
}
\end{convention*}

Theorem \ref{th:hm1} can be formulated more abstractly in terms of 
asymptotic vanishing of matrix coefficients of unitary representations.

\begin{Theorem}\label{th:hm2}
	Let $G$ be a (noncompact) connected simple Lie group with finite center
	and 	$\pi$ a unitary representation of $G$ on a Hilbert space $\mathcal{H}$.
	Then for all $v,w\in \mathcal{H}$,
	$$
	\left<\pi(g)v,w\right>\to \left<P_Gv,P_G w\right>\quad \hbox{as $g\to \infty$ in $G$, }
	$$
	where $P_G$ denotes the orthogonal projection on the subspace of the $G$-invariant vectors.
\end{Theorem}

\medskip

The study of matrix coefficients for unitary representations of semisimple Lie
groups has a long history. In particular, this subject played important role 
in the research programme of Harish-Chandra. We refer to the monographs
\cite{war1,war2,gv,knapp} for expositions of this theory.
Explicit quantitative bounds on the matrix coefficient,
which in particular imply Theorem \ref{th:hm2},
 were derived
in the works of Borel and Wallach \cite{BW}, Cowling \cite{cowling}, and 
Casselman and Milici\'c \cite{cm}. This initial approach to study of asymptotic
properties of matrix coefficients used elaborate analytic arguments that 
involved representing them as solutions of certain systems of PDE's.
Subsequently, Howe and Moore \cite{HM} developed a different approach 
to prove Theorem \ref{th:hm2} that used 
the Mautner phenomenon (cf. Theorem \ref{th:mautner} below)
and an inductive argument that derived vanishing  of matrix coefficients
on the whole group from
vanishing along a sufficiently rich collection of subgroups.
We present a version of this method here.
Other treatments of Theorems \ref{th:hm1} and \ref{th:hm2} can be also found in the monographs
\cite{zim,ht,bm}.

It is worthwhile to mention that the Howe--Moore argument \cite{HM}
is not restricted just to semisimple groups, and it gives the following
general result.
Given an irreducible unitary representation $\pi$ of a group $G$, we denote by
$$
R_\pi=\{g\in G:\, \pi(g)\in \C^\times\hbox{id}\}
$$ its projective kernel.
Since $\pi$ is unitary, it is clear that 
the matrix coefficients $|\left<\pi(g)v,w\right>|$ are constant on cosets of $R_\pi$.
One of the main results of \cite{HM} is asymptotic vanishing of matrix coefficients along $G/R_\pi$:

\begin{theo}\label{th:hm}
Let $G$ be a connected real algebraic group and $\pi$ an irreducible representation of $G$
on a Hilbert space $\mathcal{H}$. Then for any $v,w\in \mathcal{H}$,
$$
\left<\pi(g)v,w\right>\to 0\quad \hbox{as $g\to \infty$ in $G/R_\pi$.}
$$
\end{theo}

\medskip

Now we start the proof of Theorem \ref{th:hm2}. First, we note that because of the decomposition
$$
\mathcal{H}=\mathcal{H}_G\oplus \mathcal{H}_G^\perp,
$$
where $\mathcal{H}_G$ denotes the subspace of $G$-invariant vectors, it is sufficient to prove that 
for all vectors $v,w\in \mathcal{H}_G^\perp$,
$$
\left<\pi(g)v,w\right>\to 0\quad \hbox{as $g\to \infty$,}
$$
 and without loss of generality, we may assume that $\mathcal{H}$ contains no 
nonzero $G$-invariant vectors.

The proof will proceed by contradiction.
Suppose that, in contrary, 
$$
\left<\pi(g^{(n)})v,w\right>\not\rightarrow 0
$$
for some sequence $g^{(n)}\to \infty$ in $G$.
We divide the proof into four steps.

\bigskip

\noindent {\it \underline{Step 1:} Cartan decomposition.}
We shall use the Cartan decomposition for $G$:
$$
G=KA^+K,
$$
where $K$ is a maximal compact subgroup of $G$, and $A^+$ is a positive Weyl chamber 
of a Cartan subgroup of $G$.
For instance, when $G=\hbox{SL}_d(\mathbb{R})$, this decomposition holds with 
$$
K=\hbox{SO}(d)\quad\hbox{and}\quad A^+=\{\hbox{diag}(a_1,\ldots,a_d):\, a_1\ge a_2\ge \cdots\ge a_d>0\}.
$$
We write 
$$
g^{(n)}=k^{(n)} a^{(n)}\ell^{(n)}\quad\hbox{with}\;\; \hbox{$k^{(n)},\ell^{(n)}\in K$ and $a^{(n)}\in A^+$.}
$$
Since $K$ is compact, it follows that $a^{(n)}\to\infty$ in $A^+$.
Passing to a subsequence, we may arrange that the sequences $k^{(n)}$ and $\ell^{(n)}$ converge in $K$ so that, in particular,
$$
\pi(\ell^{(n)})v\to v'\quad\hbox{and}\quad \pi(k^{(n)})^*w\to  w'
$$
for some vectors $v',w'\in \mathcal{H}$.
We observe that
\begin{align*}
\left<\pi(g^{(n)})v,w\right>-\left<\pi(a^{(n)})v',w'\right>
=&\left<\pi(a^{(n)})\pi(\ell^{(n)})v,\pi(k^{(n)})^*w\right>-\left<\pi(a^{(n)})v',w'\right>\\
=&\left<\pi(a^{(n)})(\pi(\ell^{(n)})v-v'),\pi(k^{(n)})^*w\right>\\
&+
\left<\pi(a^{(n)})v',\pi(k^{(n)})^*w-w'\right>.
\end{align*}
Using that the representation $\pi$ is unitary, we deduce that
\begin{align*}
\left|\left<\pi(a^{(n)})(\pi(\ell^{(n)})v-v'),\pi(k^{(n)})^*w\right>\right|
&\le \left\| \pi(a^{(n)})(\pi(\ell^{(n)})v-v')\right\| \left\|\pi(k^{(n)})^*w\right\|\\
&= \left\| \pi(\ell^{(n)})v-v'\right\| \left\|w\right\|\to 0.
\end{align*}	
Similarly, one can show that 
$$
\left<\pi(a^{(n)})v',\pi(k^{(n)})^*w-w'\right>\to 0.
$$
Hence, we conclude that 
$$
\left<\pi(g^{(n)})v,w\right>=\left<\pi(a^{(n)})v',w'\right>+o(1),
$$
and 
$$
\left<\pi(a^{(n)})v',w'\right>\not\rightarrow 0.
$$

\bigskip

\noindent {\it \underline{Step 2:} weak convergence.}
We use the notion of `weak convergence'.
We recall that a sequence of vectors $x^{(n)}$ 
in a Hilbert space converges weakly to a vector $x$
if $\left<x^{(n)},y\right>\to \left<x,y\right>$ for all $y\in \mathcal{H}$.
We use the notation: $x^{(n)}\stackrel{w}{\longrightarrow}x$.
It is known that every bounded sequence has a weakly convergent subsequence.
In particular, it follows that, after passing to a subsequence,
we may arrange that 
$$
\pi(a^{(n)})v'\stackrel{w}{\longrightarrow}v''
$$ for some vector $v''\in \mathcal{H}$.
Then, in particular, 
$$
\left<\pi(a^{(n)})v',w'\right>\to\left<v'',w'\right>\ne 0.
$$

\bigskip

\noindent {\it \underline{Step 3:} the case when $G=\hbox{\rm SL}_2(\mathbb{R})$.}
From the previous step, we know that 
$$
\left<\pi(a^{(n)})v',w'\right>\to \left<v'',w'\right>\ne 0\quad\hbox{for $a^{(n)}=\left(\begin{tabular}{cc} $t_n$ & $0$ \\ 0 & $t_n^{-1}$\end{tabular} \right)$ with $t_n\to\infty$.}
$$

We claim that the vector $v''$ is invariant under the subgroup 
$$
U=\left\{u(s)= \left(\begin{tabular}{cc} 1 & $s$ \\ 0 & 1\end{tabular} \right):\, s\in \R\right\}.
$$
This property
will be deduced from the identity
$$
(a^{(n)})^{-1} u(s) a^{(n)}=u(s/t_n^2)\to e.
$$
One can easily check that 
$$
\pi(u(s))\pi(a^{(n)})v'\stackrel{w}{\longrightarrow}\pi(u(s))v'',
$$
so that
$$
\pi(u(s))v''= \hbox{w-lim}_{n\to \infty}\, \pi(u(s))\pi(a^{(n)})v'=\hbox{w-lim}_{n\to \infty}\,\pi(a^{(n)})\pi(u(s/t_n^2))v'.
$$
Since 
$$
\|\pi(a^{(n)})\pi(u(s/t_n^2))v'-\pi(a^{(n)})v'\|=\|\pi(u(s/t_n^2))v'-v'\|\to 0,
$$
it follows that 
$$
\hbox{w-lim}_{n\to \infty}\,\pi(a^{(n)})\pi(u(s/t_n^2))v'
=\hbox{w-lim}_{n\to \infty}\,\pi(a^{(n)})v'=v''.
$$
This proves that indeed the vector $v''$ is invariant under $U$.

Next, we show that the vector $v''$ is $G$-invariant.
We consider the function 
$$
F(g)=\left<\pi(g)v'',v''\right>\quad\hbox{ with $g\in G$.}
$$
Since $v''$ is $U$-invariant, the function $F$ is bi-invariant under $U$.
We observe that the map $gU\mapsto ge_1$ 
defines the isomorphism of the homogeneous spaces $G/U$ and $\mathbb{R}^2\backslash \{0\}$.
Hence, we may consider $F$ as a function $\mathbb{R}^2\backslash \{0\}$.
Since the $U$-orbits in $\mathbb{R}^2$ are the lines $y=c$ with $c\ne 0$ and 
the points $(x,0)$, we conclude that $F$ is constant on each line $y=c$ with $c\ne 0$.
By continuity, it follows that $F$ is also constant on the line $y=0$.
For $a(t)=\left(\begin{tabular}{cc} $t$ & $0$ \\ 0 & $t^{-1}$\end{tabular} \right)$,
$$
\left<\pi(a(t))v'',v''\right>=F(a(t)e_1)= F(te_1)=F(e_1)=\|v''\|^2.
$$
Since this gives the equality in the Cauchy--Schwarz inequality,
the vectors $\pi(a(t)) v''$ and $v''$ must be colinear, and
we deduce that $\pi(a(t))v''=v''$, so that the vector $v''$
is also invariant under the subgroup $A=\{a(t)\}$.
Hence, the function $F$ is also constant on $AU$-orbits in $\mathbb{R}^2$.
Since the half-spaces $\{y>0\}$ and $\{y<0\}$ are single $AU$-orbits,
It follows from the continuity of $F$, that this function is identically constant,
that is,
$$
F(g)=\left<\pi(g)v'',v''\right>=\|v''\|^2\quad\hbox{for all $g\in G$.}
$$
This gives the equality in the Cauchy--Schwarz inequality, and as before
we deduce that $\pi(g)v''=v''$ for all $g\in G$. However, we have assumed that 
there is nonzero $G$-invariant vectors. This gives a contradiction, and
completes the proof of the theorem in the case $G=\hbox{SL}_2(\R)$.

\medskip

We note that the above argument, in fact, implies 
the following ``Mautner property'' of unitary representations of $\hbox{SL}_2(\R)$:
every $U$-invariant vector is always $\hbox{SL}_2(\R)$-invariant.
More generally, one says that a closed subgroup $H$ of topological group $G$
has {\it Mautner property} if for every unitary representation of $G$,
$H$-invariant vectors are also invariant under $G$. Subgroups
satisfying this property have appeared in a work of Segal and von Neumann \cite{seg_new},
and Mauntner \cite{mau} used this phenomenon to study ergodicity of the geodesic flow 
on locally symmetric spaces. The following general version of the Mautner property
was established by Moore \cite{m0}:

\begin{theo}\label{th:mautner}
Let $G$ be a (noncompact) simple connected Lie group with finite center.
Then every noncompact closed subgroup of $G$ has the Mautner property.
\end{theo}

Subsequently, more general versions of this result were proved
by Moore \cite{m00}, Wang \cite{wang1,wang2}, and Bader, Furman, Gorodnik, Weiss \cite{bfgw}.

\bigskip

\noindent {\it \underline{Step 4:} inductive argument.}
Our next task is to develop an inductive argument which 
allows to deduce asymptotic vanishing of matrix coefficients
using vanishing along smaller subgroup.
We give a complete proof when $G=\hbox{SL}_d(\R)$,
but similar ideas can be also extended to general semisimple Lie groups
using their structure theory.

For $a\in A^+$, the set of simple roots is given by
$$
\alpha_i(a)=a_i/a_{i+1}\quad\hbox{ for $i=1,\ldots,d-1$.}
$$
The functions $\alpha_i$, $i=1,\ldots,d-1$, provide a coordinate system on $A^+$.
Given a sequence $a^{(n)}\in A^+$ such $a^{(n)}\to \infty$, we have 
$\max_i \alpha_i(a^{(n)})\to \infty$.
After passing to a subsequence, we may assume that $\alpha_i(a^{(n)})\to \infty$
for some $i$. We introduce the subgroup
$$
U_i=\left\{\left(
\begin{tabular}{cc} 
$I_i$ & $u$ \\
$0$ & $I_{d-i}$
\end{tabular}
\right):\, u\in \hbox{M}_{i, d-i}(\mathbb{R}) \right\}.
$$
For $a\in A^+$,
$$
a^{-1}\left(
\begin{tabular}{cc} 
$I_i$ & $(u_{lk})$ \\
$0$ & $I_{d-i}$
\end{tabular}
\right) a=
\left(
\begin{tabular}{cc} 
$I_i$ & $\left(\frac{a_k}{a_l}u_{lk}\right)$ \\
$0$ & $I_{d-i}$
\end{tabular}
\right).
$$
Since for $l\le i <k$,
$$
\frac{a_l}{a_k}=\frac{a_l}{a_{l+1}}\cdots \frac{a_{k-1}}{a_k}\ge \frac{a_i}{a_{i+1}}=\alpha_i(a),
$$
it follows that $a^{(n)}_l/a^{(n)}_k\to\infty$. Hence, for $g\in U_i$,
$$
(a^{(n)})^{-1} ga^{(n)}\to e.
$$
Using this property, we may argue exactly as in Step 3 to conclude that
the vector $v''$ is $U_i$-invariant.
For $1\le l\le i$ and $i+1\le k\le d$,
we denote by $U_{lk}$ the corresponding one-parameter unipotent subgroup of $U_i$.
We observe that $U_{lk}$ can be embedded in an obvious way as a subgroup 
of the group $G_{lk}\simeq \hbox{SL}_2(\R)$ contained in $G$.
Since the vector $v''$ is invariant under $U_{lk}$,
it follows from Step 3 that it is also invariant under $G_{lk}$
when $1\le l\le i$ and $i+1\le k\le d$.
Finally, we check that these groups  $G_{lk}$ generate $G=\hbox{SL}_d(\R)$,
so that the vector $v''$ is $G$-invariant.
This gives a contradiction and completes the proof of the theorem.

\section{Application: counting lattice points}\label{sec:counting}

Given a lattice $\Lambda$ in the Euclidean space $\mathbb{R}^d$ and 
a Euclidean ball $B$ in $\mathbb{R}^d$, one can show using a simple geometric argument
that 
$$
|\Lambda\cap B|\sim \frac{\vol(B)}{\vol(\mathbb{R}^d/\Lambda)}\quad \hbox{as $\vol(B)\to\infty.$}
$$
This result also holds for more general families of domains satisfying some regularity
assumptions. The analogous lattice counting problem for the hyperbolic space is more difficult
because of the exponential volume growth of the hyperbolic balls, and proving an asymptotic
formula even without an error term requires analytic tools.


The hyperbolic lattice point counting problem was studied by Delsarte \cite{del},
Huber \cite{hub}, and Patterson \cite{pat}. 
These works used spectral expansion of the counting functions
in terms of the eigenfunctions of the Laplace-Beltrami operator.
Margulis  in his PhD thesis \cite{mar0,mar} discovered that the lattice point  counting problem on manifolds of variable negative curvature
can be solved using solely the mixing property.
Bartels \cite{bart} proved an asymptotic formula
for the number of lattice points in connected semisimple Lie groups
using a version of Theorem \ref{th:hm1}.
Subsequently, this approach was generalised to counting 
lattice orbit points on affine symmetric varieties
by Duke, Rudnick, Sarnak \cite{drs} and Eskin, Mcmullen \cite{em}.
We refer to the survey \cite{bab} for a comprehensive discussion of
the lattice point counting problems. 

\medskip

In this section, we consider a more general counting 
problem for lattice points in locally compact groups.
In particular, we prove the following result:

\begin{Theorem}\label{th:count0}
Let $G$ be a (noncompact) connected simple matrix Lie group with finite centre, and the sets
\begin{equation}\label{eq:norm_ball}
B_t=\{g\in G:\, \|g\|<t \}
\end{equation}
are defined by a norm on the space of matrices.
Then for any lattice subgroup $\Gamma$ of $G$,
$$
|\Gamma\cap B_t|\sim \frac{\hbox{\rm vol}(B_t)}{\hbox{\rm vol}(G/\Gamma)}\quad \hbox{as $t\to\infty$.}
$$
\end{Theorem}

More generally, we establish an asymptotic counting formula in a setting of locally compact groups satisfying a certain mixing assumption.
Let $G$ be a locally compact second countable group and $\Gamma$ a lattice subgroup in $G$.
We fix a Haar measure $m$ on $G$ which also induced the measure $\mu$
on the factor space $X=G/\Gamma$ by
$$
\int_{G/\Gamma} \Big(\sum_{\gamma\in \Gamma}\psi(g\gamma)\Big)\,d\mu(g\Gamma)=\int_G \psi \, dm
\quad\hbox{for $\psi\in C_c(G/\Gamma).$}
$$
We normalise the measure $m$ so that $\mu(X)=1$. Then we obtain a continuous 
measure-preserving action of $G$ on the probability space $(X,\mu)$.
 
We say that a family of bounded measurable sets $B_t$ in $G$ is {\it well-rounded} (cf. \cite{em}) if
for every $\delta>1$, there exists a symmetric neighbourhood $\mathcal{O}$ of 
identity in $G$ such that 
\begin{equation}
\label{eq:well_round}
\delta^{-1}\, m\left(\bigcup_{g_1,g_2\in \mathcal{O}} g_1 B_t g_2 \right)\le m(B_t)\le \delta\, m\left(\bigcap_{g_1,g_2\in \mathcal{O}} g_1 B_t g_2 \right)
\end{equation}
for all $t$.

\begin{Theorem}\label{th:count}
Let $G$ be a locally compact second countable group, and let $B_t$  be a family of 
well-rounded compact sets in $G$ such that  $m(B_t)\to \infty$.
Let $\Gamma$ be a lattice subgroup in $G$ such that the action of $G$ on the space $G/\Gamma$ is mixing.
Then
$$
|\Gamma\cap B_t|\sim m(B_t)\quad \hbox{as $t\to\infty$.}
$$
\end{Theorem}

It follows from a Fubini-type argument that the  action of $G$ on $X=G/\Gamma$
is ergodic (i.e., every almost everywhere invariant function is constant almost everywhere).
Hence, when $G$ is a connected simple  Lie group with finite centre,
it follows from Theorem \ref{th:hm1} that the actions of $G$ on $X=G/\Gamma$ is mixing of order two. One can also check that the regularity condition \eqref{eq:well_round}
is satisfied for the norm balls \eqref{eq:norm_ball} (see \cite{drs},\cite{em}).
Hence, Theorem \ref{th:count0} follows from Theorem \ref{th:count}.

\medskip

We start the proof of Theorem \ref{th:count}
by realising the counting function as a function on the homogeneous space $X=G/\Gamma$.
We set 
\begin{equation}
\label{eq:ft}
F_t(g_1,g_2)=\sum_{\gamma\in\Gamma} \chi_{B_t}(g_1\gamma g_2^{-1}).
\end{equation}
In the first part of the argument,
we do not impose any regularity assumptions
on the compact domains $B_t$ and just assume that $m(B_t)\to\infty$.
Since 
$$
F_t(g_1\gamma_1,g_2\gamma_2)=F_t(g_1,g_2)\quad\hbox{for all $g_1,g_2\in G$ and $\gamma_1,\gamma_2\in\Gamma$,}
$$
$F_t$ defines a function on $G/\Gamma\times G/\Gamma$.
We note that $F_t(e,e)=|\Gamma\cap B_t|$, so that it remains to investigate
the asymptotic behaviour of $F_t$ at the identity coset.

The crucial connection between the original counting problem and estimating correlations
is provided by the following computation.
For a real-valued test-function $\phi\in C_c(G/\Gamma)$, we obtain that
\begin{align*}
\left<F_t,\phi\otimes\phi\right>
&=\int_{G/\Gamma\times G/\Gamma} F_t(g_1,g_2)\phi(g_1)\phi(g_2)\, d\mu(g_1\Gamma)d\mu(g_2\Gamma)\\
&=\int_{G/\Gamma\times G/\Gamma} \left(\sum_{\gamma\in\Gamma} \chi_{B_t}(g_1\gamma g_2^{-1})\right) \phi(g_1)\phi(g_2)\, d\mu(g_1\Gamma)d\mu(g_2\Gamma)\\
&=\int_{G/\Gamma\times G/\Gamma} \left(\sum_{\gamma\in\Gamma} \chi_{B_t}(g_1(g_2\gamma)^{-1})\right) \phi(g_1)\phi(g_2)\, d\mu(g_1\Gamma)d\mu(g_2\Gamma)\\
&=\int_{G/\Gamma\times G} \chi_{B_t}(g_1g_2^{-1}) \phi(g_1)\phi(g_2)\, d\mu(g_1\Gamma)dm(g_2).
\end{align*}
We denote by $\pi$ the unitary representation of $G$ on $L^2(G/\Gamma)$ defined as in \eqref{eq:rep}. Using a change of variables $b=g_1g_2^{-1}$, we deduce that
\begin{align*}
	\left<F_t,\phi\otimes\phi\right>
	&=\int_{G/\Gamma\times G} \chi_{B_t}(b) \phi(g_1)\phi(b^{-1}g_1)\, d\mu(g_1\Gamma)dm(b)\\
	&= \int_{B_t} \left<\pi(b)\phi,\phi\right>\, dm(b).
\end{align*}
According to our assumption,  
$$
\left<\pi(b)\phi,\phi\right>\longrightarrow \left(\int_{G/\Gamma} \phi\, d\mu\right)^2\quad \hbox{as $b\to\infty$.}
$$
Hence, since $m(B_t)\to \infty$, it follows that
\begin{equation}
\label{eq:weak}
\frac{\left<F_t,\phi\otimes\phi\right>}{m(B_t)}=
\frac{1}{m(B_t)}\int_{B_t} \left<\pi(b)\phi,\phi\right>\, dm(b)
\longrightarrow \left(\int_{G/\Gamma} \phi\, d\mu\right)^2
\end{equation}
as $t\to\infty$.
We note that \eqref{eq:weak} holds for any functions $F_t$ defined in terms of compact subsets
$B_t$ such that $m(B_t)\to\infty$.

\medskip

Our next task is to upgrade the weak convergence of functions $F_t$ established in \eqref{eq:weak} to the pointwise convergence. For this step, we use the regularity assumption
\eqref{eq:well_round} on the domains $B_t$.
We take any $\delta>1$ and choose the neighbourhood $\mathcal{O}$ of identity in $G$
as in \eqref{eq:well_round}. We set
$$
B_t^+=\bigcup_{g_1,g_2\in \mathcal{O}} g_1 B_t g_2
\quad\hbox{and}\quad
B_t^-=\bigcap_{g_1,g_2\in \mathcal{O}} g_1 B_t g_2,
$$
and consider the corresponding functions $F_t^+$ and $F_t^-$ defined as in \eqref{eq:ft}.
It follows from \eqref{eq:well_round} that
\begin{equation}\label{eq:well2}
\delta^{-1}\, m(B_t^+)\le m(B_t)\le \delta\, m(B_t^-).
\end{equation}
In particular, $m(B_t^\pm)\to \infty$.

We take a nonnegative function $\tilde \phi\in C_c(G)$ such that 
$$
\hbox{supp}(\tilde \phi)\subset \mathcal{O}\quad\hbox{and}\quad \int_G\tilde \phi\, dm=1,
$$
and define 
a function $\phi\in C_c(G/\Gamma)$ as $\phi(g)=\sum_{\gamma\in\Gamma} \tilde \phi(g\gamma)$.
Then
\begin{align*}
\left<F^+_t,\phi\otimes \phi\right>&=\int_{G/\Gamma\times G/\Gamma} F^+_t(g_1,g_2)\phi(g_1)\phi(g_2)\, d\mu(g_1\Gamma)d\mu(g_2\Gamma)\\
&=\int_{G/\Gamma\times G/\Gamma} F^+_t(g_1,g_2) \left( \sum_{\gamma_1,\gamma_2\in\Gamma} \tilde \phi(g_1\gamma_1)\tilde \phi(g_2\gamma_2)\right)\, d\mu(g_1\Gamma)d\mu(g_2\Gamma)\\
&=\int_{G/\Gamma\times G/\Gamma}  \left( \sum_{\gamma_1,\gamma_2\in\Gamma} F^+_t(g_1\gamma_1,g_2\gamma_2)\tilde \phi(g_1\gamma_1)\tilde \phi(g_2\gamma_2)\right)\, d\mu(g_1\Gamma)d\mu(g_2\Gamma)\\
&=\int_{G\times G}  F^+_t(g_1,g_2)\tilde \phi(g_1)\tilde \phi(g_2)\, dm(g_1)dm(g_2).
\end{align*}
We observe that when $g_1,g_2\in\mathcal{O}$,
$$
F^+_t(g_1,g_2)=\sum_{\gamma\in \Gamma} \chi_{g_1^{-1}B^+_t g_2}(\gamma)\ge 
\sum_{\gamma\in \Gamma} \chi_{B_{t}}(\gamma)=|\Gamma\cap B_{t}|,
$$
so that since $\hbox{supp}(\tilde \phi)\subset \mathcal{O}$, we obtain that
$$
\left<F^+_t,\phi\otimes \phi\right>\ge |\Gamma\cap B_t| \left(\int_G \tilde \phi \, dm\right)^2\ge |\Gamma\cap B_{t}|.
$$
Hence, it follows from \eqref{eq:weak} and \eqref{eq:well2} that
$$
\limsup_{t\to\infty} \frac{|\Gamma\cap B_{ t}|}{m(B_t)}\le 
\delta\limsup_{t\to\infty} \frac{\left<F^+_{t},\phi\otimes \phi\right>}{m(B^+_t)}
=\delta
$$
for all $\delta>1$. 
A similar argument applied to the function $F_t^-$ gives
$$
\liminf_{t\to\infty} \frac{|\Gamma\cap B_{ t}|}{m(B_t)}\ge 
\delta^{-1}\liminf_{t\to\infty} \frac{\left<F^-_{t},\phi\otimes \phi\right>}{m(B^-_t)}
=\delta^{-1}
$$
for all $\delta>1$. This implies Theorem \ref{th:count}.

\medskip

It is worthwhile to mention that 
Gorodnik and Nevo \cite{gn1,gn2} showed 
the asymptotic formula for counting lattice points can be deduced solely from 
an ergodic theorem for averages along the sets $B_t$ on the space $X=G/\Gamma$. `Ergodic theorem' is a much more prolific phenomenon than `mixing property'.

\section{Quantitative estimates on matrix coefficients}\label{sec:cor_2}

The goal of this section is to establish quantitative estimates on 
matrix coefficients for unitary representations $\pi$ of 
higher-rank simple groups  $G$ (for instance, for $G=\hbox{SL}_d(\R)$
with $d\ge 3$). It is quite remarkable that this quantitative bound 
for higher-rank groups holds uniformly for all representations without invariant vectors. 

A qualitative bound on matrix coefficients may only hold on a proper subset of vectors,
and to state such a bound, we introduce a notion of $K$-finite vectors.
Let $K$ be a maximal compact subgroup of $G$.
By the Peter--Weyl Theorem, 
a unitary representation $\pi|_K$ 
splits as a sum of finite-dimensional irreducible representations.
A vector $v$ is called {\it $K$-finite} if the span of $\pi(K)v$ is of 
finite dimension. 
We set 
$$
d_K(v)=\dim \left<\pi(K)v\right>.
$$
The space of $K$-finite vectors is dense in the represenation space.

With this notation, we prove:

\begin{Theorem}
	\label{th:high-rank}
	Let $G$ be a (noncompact) connected   simple higher-rank matrix Lie group with finite centre
	and $K$ a maximal compact subgroup of $G$.
	Then there exist $c,\delta>0$ such that for any unitary representation $\pi$ of 
	$G$ on a Hilbert space $\mathcal{H}$ without nonzero $G$-invariant vectors,
	the following estimate holds: for all elements $g\in G$ and all $K$-finite vectors $v,w\in \mathcal{H}$,
	$$
	|\left<\pi(g)v,w\right>|\le c\,  d_K(v)^{1/2}d_K(w)^{1/2}\|v\|\|w\|\, \|g\|^{-\delta}.
	$$
\end{Theorem}

As we already remarked in Section \ref{sec:decay}, 
asymptotic properties of matrix coefficients for semisimple Lie groups
has been studied extensively starting with foundational works of Harish-Chandra
(see, for instance, \cite{hc}).
Explicit quantitative bounds on matrix coefficients
have been obtained, in particular, in the works of 
Borel and Wallach \cite{BW}, Cowling \cite{cowling}, Howe \cite{h},
Casselman and Milici\'c \cite{cm},
Cowling, Haagerup, and Howe \cite{chh},
Li \cite{li}, Li and Zhu \cite{li2}, and Oh \cite{oh1,oh2}. 
Here we follow the elegant elementary approach of Howe and Tan \cite{ht}
to prove Theorem \ref{th:high-rank}.

\medskip

We start our investigation by analysing the unitary representations of the semidirect product
$$
L=\hbox{SL}_2(\R)\ltimes \R^2.
$$
We shall use the following notation:
\begin{align}
S&=\hbox{SL}_2(\R), \nonumber \\
K_S&=\hbox{SO}(2)=\left\{k(\theta)=
\left( \begin{tabular}{cc} $\cos\theta$ & $-\sin\theta$ \\ $\sin\theta$ & $\cos\theta$ \end{tabular}\right):\, \theta\in [0,2\pi)\right\}, \nonumber\\
A_S&=\left\{a(t)=\left( \begin{tabular}{cc} $t$ & $0$ \\ $0$ & $t^{-1}$ \end{tabular}\right):\, t>0\right\},\label{eq:notation}\\
U_S&=\left\{u(s)=\left(\begin{tabular}{cc} 1 & $s$ \\ 0 & 1\end{tabular} \right):\,
s\in \R\right\}.\nonumber
\end{align}

\begin{Prop}\label{p:howe}
Let $\pi$ be a unitary representation of $L$ on the Hilbert space $\mathcal{H}$
that does not have any nonzero $\mathbb{R}^2$-invariant vectors.
Then for all vectors $v, w$ belonging to a $K_S$-invariant dense subspace of $\mathcal{H}$,
\begin{equation}
\label{eq:bound0}
|\left<\pi(a(t))v,w\right>|\le c(v,w)\, t^{-1} \quad \hbox{when $t\ge 1$.}
\end{equation}
\end{Prop}

\begin{proof}
We consider the restricted representation $\pi|_{\mathbb{R}^2}$
which can be decomposed with respect irreducible one-dimensional unitary representations of $\R^2$
--- the unitary characters of $\R^2$:
$$
r\mapsto \chi_z(r)=e^{i\left<z,r\right>},\quad z\in\R^2.
$$
Namely, there exists a Borel projection-valued measure $P$ on $\R^2$ such that
$$
\pi(r)=\int_{\R^2}\chi_z(r)\, dP_z\quad\hbox{ for $r\in \R^2$. }
$$
We shall use that the measure $P$ satisfies an equivariance property with respect
to the action of $S$: since 
$$
\pi(g)\pi(r)\pi(g)^{-1}=\pi(g(r))\quad\hbox{for $g\in S$ and $r\in \R^2$,}
$$
and
$$
\pi(g(r))=\int_{\R^2}\chi_z(g(r))\, dP_z=\int_{\R^2}\chi_{g^{t}(z)}(r)\, dP_z,
$$
it follows that 
\begin{equation}
\label{eq:equiv}
\pi(g)P_B\pi(g)^{-1}=P_{(g^t)^{-1}B}\quad \hbox{ for Borel $B\subset \R^2$ and $g\in S$.}
\end{equation}
For $s>1$, we set
$$
\Omega_s=\{r\in\R^2:\, s^{-1}\le \|r\|\le s \},
$$
and consider the closed subspace $\mathcal{H}_s=\hbox{Im}(P_{\Omega_s})$.
Since the set $\Omega_s$ is invariant under $K_S$, it follows 
from \eqref{eq:equiv} that the subspace $\mathcal{H}_s$ is $K_S$-invariant.
Using that the projection-valued measure $P$ is strongly continuous, 
we deduce that for all $v\in \mathcal{H}$,
$$
P_{\Omega_s}v\to P_{\R^2\backslash \{0\}}v=v-P_{0}v\quad\hbox{ as $s\to\infty$.}
$$
Moreover, since we assumed that there is no nonzero $\R^2$-invariant vectors, $P_0=0$.
This shows that $\cup_{s>1} \mathcal{H}_s$ is dense in $\mathcal{H}$.
Using that the span of $K_S$-eigenvectors is dense in $\mathcal{H}_s$,
it remains to show that \eqref{eq:bound0} holds
for all $K_S$-eigenvectors in $\mathcal{H}_s$ with $s>1$.
We recall that  $P_{B_1}P_{B_2}=P_{B_1\cap B_2}$ for Borel $B_1,B_2\subset \R^2$.
This allows us to compute that for $v,w\in \mathcal{H}_s$,
\begin{align*}
\left<\pi(a(t))v,w\right>&=\left<\pi(a(t))P_{\Omega_s}v, P_{\Omega_s}w\right>=
\left<P_{a(t)^{-1}\Omega_s}\pi(a(t))v, P_{\Omega_s}w\right>\\
&=\left<P_{a(t)^{-1}\Omega_s}\pi(a(t))v, P_{a(t)^{-1}\Omega_s}P_{\Omega_s}w\right>\\
&=\left<P_{a(t)^{-1}\Omega_s}\pi(a(t))v, P_{a(t)^{-1}\Omega_s\cap\Omega_s}w\right>\\
&=\left<P_{a(t)^{-1}\Omega_s\cap \Omega_s}\pi(a(t))v, P_{a(t)^{-1}\Omega_s\cap\Omega_s}w\right>\\
&=\left<\pi(a(t)) P_{\Omega_s\cap a(t)\Omega_s}v, P_{a(t)^{-1}\Omega_s\cap\Omega_s}w\right>.
\end{align*}	
Hence, by the Cauchy--Schwarz inequality, 
\begin{align}\label{eq:cs}
|\left<\pi(a(t))v,w\right>|\le 
\|P_{\Omega_s\cap a(t)\Omega_s}v\|\, \|P_{a(t)^{-1}\Omega_s\cap\Omega_s}w\|.
\end{align}
We observe that the region $\Omega_s\cap a(t)\Omega_s$ is contained in two sectors of angle 
$$
\alpha\le 2\sin^{-1}(s^2/t)
$$
(see Figure \ref{f:ellip}). 
\begin{figure}[b]
	\includegraphics[width=0.6\linewidth]{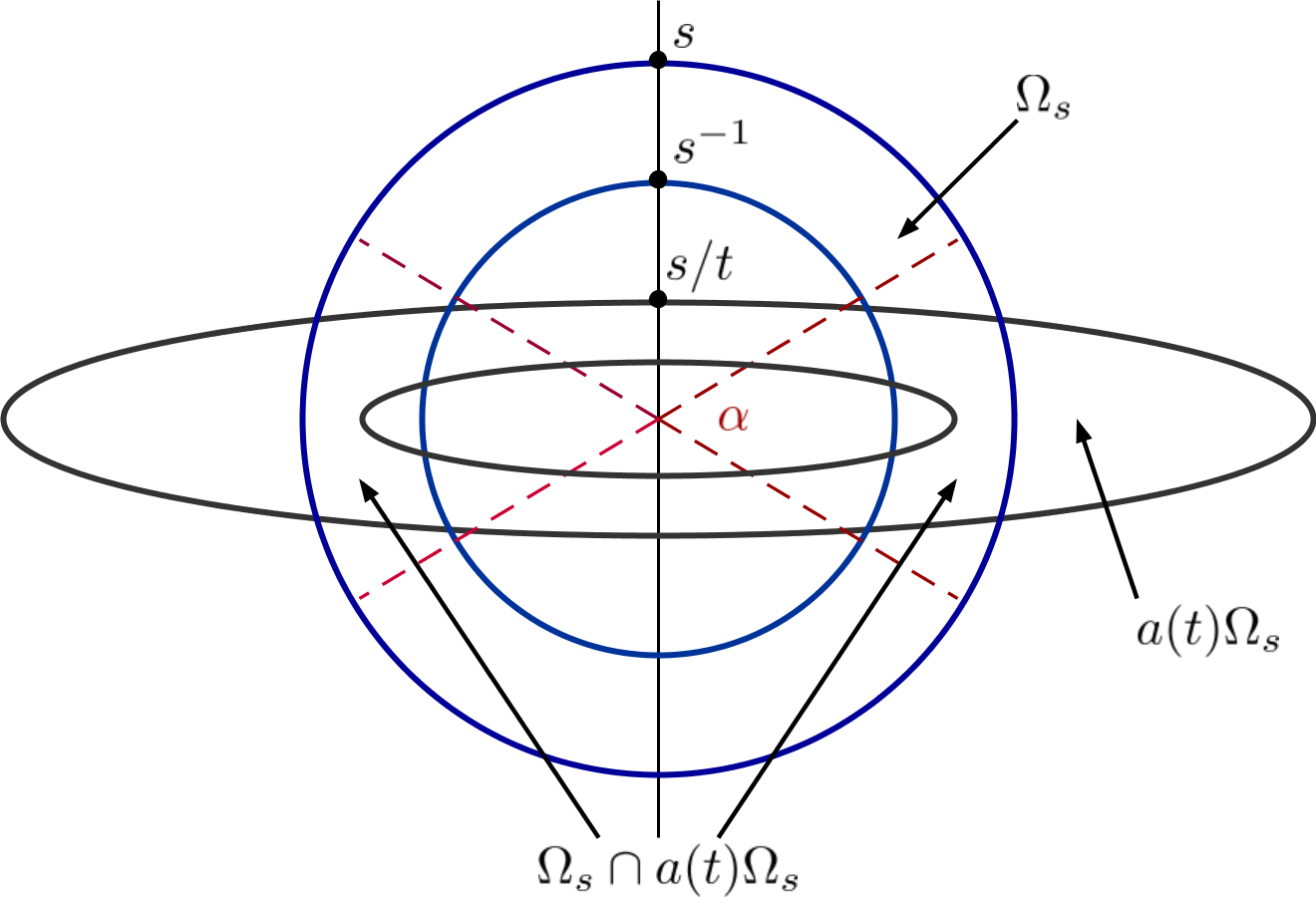} \label{f:ellip}
	\caption{Estimating the angle $\alpha$.} \label{f:ellip}
\end{figure}
We take $\theta_m=2\pi/m$ such that $\theta_{m+1}\le\alpha<\theta_m$,  and consider the partition
\begin{equation}
\label{eq:decomp0}
\R^2\backslash \{0\}=\bigsqcup_{i=1}^m S_i
\end{equation}
into sectors such that 
\begin{equation}
\label{eq:contain}
\Omega_s\cap a(t)\Omega_s\subset S_1.
\end{equation}
We note that $k_{\theta_m}(S_i)=S_{(i+1)\,\hbox{\tiny mod}\, m}$.
Now we suppose that $v$ is an eigenvector of $K_S$, that is, 
$\pi(k_\theta)v=e^{i\lambda\theta}v$ for some $\lambda\in \R$. Then
using \eqref{eq:equiv}, we deduce that
$$
\pi(k_{\theta_m})P_{S_i}v=P_{k_{\theta_m}S_i}\pi(k_{\theta_m})v=e^{i\lambda\theta_m} P_{S_{(i+1)\,\hbox{\tiny mod}\, m}}v.
$$
Hence, it follows that 
$$
\|P_{S_i}v\|=\|P_{S_{(i+1)\,\hbox{\tiny mod}\, m}}v\|.
$$	
For \eqref{eq:decomp0}, we obtain
the orthogonal decomposition
$$
v=P_{\R^2\backslash \{0\}}v=\sum_{i=1}^m P_{S_i} v,
$$
so that 
$$
\|v\|^2=\sum_{i=1}^m \|P_{S_i}v\|^2\quad\hbox{and}\quad \|P_{S_i}v\|=m^{-1/2}\|v\|.
$$
It follows from the inclusion \eqref{eq:contain} that
$$
\|P_{\Omega_s\cap a(t)\Omega_s}v\|\le \|P_{S_1}v\|\ll \left(\sin^{-1}(s^2/t)\right)^{1/2} \|v\|\ll_s t^{-1/2}\|v\|.
$$
A similar argument also gives that when $w$ is an eigenvector of $K_S$,
$$
\|P_{a(t)^{-1}\Omega_s\cap\Omega_s}w\|\ll_s t^{-1/2}\|w\|.
$$
Finally, we conclude  from \eqref{eq:cs} that 
$$
|\left<\pi(a(t))v,w\right>|\ll_s t^{-1}\|v\|\|w\|,
$$
which proves the proposition.
\end{proof}	

The above argument has been generalised by Konstantoulas \cite{konst} to give 
bounds on correlations of higher orders.

\medskip

Proposition \ref{p:howe} gives the best possible bound in terms of the parameter $t$,
but the drawback is that the dependence on the vectors $v,w$ is not explicit. Our goal will
be to derive a more explicit estimate for matrix coefficients. 
We observe that Proposition \ref{p:howe} implies an integrability estimate
on the functions $g\mapsto \left<\pi(g)v,w\right>$, $g\in G$. 
This will eventually allow us to reduce our study to the case of the regular representation.
We recall that the invariant
measure on $S=\hbox{SL}_2(\R)$ 
with respect to the Cartan decomposition $S=K_SA_SK_S$ is given by
$$
\int_S f(g)\,dg=\int_{[0,2\pi)\times [1,\infty)\times [0,2\pi)} f(k({\theta_1})a(t)k({\theta_2}))(t^2-t^{-2})\, d\theta_1 \frac{dt}{t}d\theta_2
$$
for $f\in C_c(G)$. Hence, from Proposition \ref{p:howe}, we deduce that

\begin{Cor}\label{c:integral}
With the notation as in Proposition \ref{p:howe},
for vectors $v,w$ belonging to a dense $K_S$-invariant subspace of $\mathcal{H}$,
the functions $g\mapsto \left<\pi(g)v,w\right>$, $g\in S$, are $L^{2+\epsilon}$-integrable
for all $\epsilon>0$.
\end{Cor}

We say that a unitary representation $\pi$ of a group $G$ is {\it $L^p$-integrable}
if the matrix coefficients $g\mapsto \left<\pi(g)v,w\right>$, $g\in G$,
belong to $L^p(G)$ for vectors $v,w$ from a dense subset.

\medskip

The following result concerns representations of a general locally compact group $G$.
We define the {\it regular representation} $\lambda_G$  on $L^2(G)$ by
$$
\lambda_G(g)\phi(x)=\phi(xg)\quad \hbox{for $\phi\in L^2(G).$}
$$

\begin{Prop}\label{p:l2}
Let $\rho$ be a unitary representation of $G$ on a Hilbert space $\mathcal{H}$.
We assume that the functions $g\mapsto \left<\rho(g)v,w\right>$	
 belong to $L^2(G)$ for vectors $v,w$ belonging to a dense subspace of $\mathcal{H}$.
 Then there exists an isometric embedding 
 $$
 \mathcal{I}:\mathcal{H}\to \oplus_{n\ge 1} L^2(G)
 $$
 such that
 for $g\in G$, we have 
 \[\xymatrixcolsep{6pc}\xymatrix{
 	\mathcal{H} \ar[d]_{\rho(g)} \ar[r]^{\mathcal{I}} & {\bigoplus}_{n\ge 1} L^2(G) \ar[d]^{\oplus_{n\ge 1}\lambda_G(g)}\\
 	\mathcal{H} \ar[r]^{\mathcal{I}} & \bigoplus_{n\ge 1} L^2(G)
 }\]
 \end{Prop}

\begin{proof}
Let $\mathcal{H}_0$ be a countable orthonormal dense subset of $\mathcal{H}$ such that 
the functions $g\mapsto \left<v,\rho(g)w\right>$ belong to $L^2(G)$ for all $v,w\in \mathcal{H}_0$. 
For $v,w\in \mathcal{H}$, we set
$$
f_{v,w}(x)=\left<v,\rho(x)w\right>,\quad x\in G,
$$
and define the map
$$
\mathcal{I}:\left<\mathcal{H}_0\right> \to {\bigoplus}_{n\ge 1} L^2(G): \, w\mapsto \left(f_{v,w}: v\in\mathcal{H}_0\right).
$$
Since 
$$
\lambda_G(g)f_{v,w}=f_{v,\rho(g)w}\quad\hbox{for $g\in G$ and $v,w\in \mathcal{H}$,}
$$
we conclude that
$$
\mathcal{I}\circ \rho(g)=\big(\oplus_{n\ge 1}\lambda_G(g)\big)\circ \mathcal{I}\quad\hbox{for $g\in G$.}
$$
Moreover, since $\mathcal{H}_0$ forms an orthonormal basis of $\mathcal{H}$,
$$
\sum_{v\in\mathcal{H}_0} \|f_{v,w}\|_{2}^2= \|w\|^2\quad\hbox{for $w\in \mathcal{H}_0$,}
$$
it follows that 
$$
\|\mathcal{I} w\|=\|w\|\quad \hbox{for $w\in \mathcal{H}_0$}.
$$
Hence, one can check that $\mathcal{I}$ extends to an isometric embedding, as required.
\end{proof}	

Although Proposition \ref{p:l2} can not be directly applied to the representation $\pi$ appearing in Proposition \ref{p:howe},
we deduce the following corollary about its tensor square $\pi\otimes \pi$.

\begin{Cor}\label{c:embed}
Let $\pi$ be a unitary representation of $L$ as in Proposition \ref{p:howe}.
Then the representation $(\pi\otimes \pi)|_S$ embeds as 
a suprepresentation of $\bigoplus_{n\ge 1} \lambda_S$.
\end{Cor}

\begin{proof}
According to Corollary \ref{c:integral}, the functions
$g\mapsto \left<\pi(g)v,w\right>$, $g\in S$, are $L^p$-integrable for all $p>2$
when $v,w$ belongs to a suitable orthonormal basis $\mathcal{H}_0$ of $\mathcal{H}$.
Then it follows from the Cauchy--Schwarz inequality that for 
$v_1,v_2,w_1,w_2\in \mathcal{H}_0$, the functions
$$
\left<(\pi\otimes \pi)(g)(v_1\otimes v_2),w_1\otimes w_2\right>
=\left<\pi(g)v_1,w_1\right>
\left<\pi(g)v_2,w_2\right>
$$
belong to $L^2(S)$. Hence, the claim is implied by Proposition \ref{p:l2}.
\end{proof}

The above result ultimately reduces our original problem 
regarding representations $\pi|_S$ to the study of matrix coefficients
for the regular representation $\lambda_S$. It turns out that the matrix coefficients of the latter representation can be estimated in terms of an explicit
function that we now introduce. In fact, this is true 
for general connected semisimple Lie groups $G$. We 
recall that in this setting, there is the {\it Iwasawa decomposition}
$$
G=UAK,
$$
where $K$ is a maximal compact subgroups of $G$, $A$ is a Cartan subgroup, and
$U$ is the subgroup generated by positive root subgroups.
The invariant measure with respect to the Iwasawa decomposition is given by
\begin{equation}
\label{eq:iwa_measure}
\int_G f(g)\, dg=\int_{U\times A\times K} f(uak)\Delta(a)\, du da dk\quad\hbox{for $f\in C_c(G)$,}
\end{equation}
where $\Delta$ denotes the modular function of the group $UA$, and 
$du$, $da$, and  $dk$ denote the invariant measures on the corresponding factors.
For example, for the group $S=\hbox{SL}_2(\R)$, using the notation \eqref{eq:notation},
we have the the Iwasawa decomposition $S=U_SA_SK_S$, and the modular function is
given by $\Delta(a(t))=t^{-2}$.
The product map $U\times A\times K\to G$ defines a diffeomorphism,
and for $g\in G$, we denote by ${\sf u}(g)\in G$,  ${\sf a}(g)\in A$, and ${\sf k}(g)\in K$
the unique elements such that 
$$
g={\sf u}(g){\sf a}(g){\sf k}(g).
$$
The {\it Harish-Chandra function} is defined as 
$$
\Xi(g)=\int_{K} \Delta ({\sf a}(kg))^{-1/2}\, dk \quad\hbox{ for  $g\in G$.}
$$
It is easy to check that the function $\Xi$ is bi-$K$-invariant.

In the case when $S=\hbox{SL}_2(\R)$, the Harish-Chandra function 
 can be explicitly computed as
$$
\Xi(a(t))=\frac{1}{2\pi}\int_{0}^{2\pi}(t^{-2}\cos^2\theta+ t^2 \sin^2\theta)^{-1/2}\, d\theta.
$$ 
Moreover, one can check for all $\epsilon>0$
\begin{equation}
\label{eq:xi_b}
\Xi(a(t))\ll_\epsilon t^{-1+\epsilon}\quad\hbox{when $t\ge 1$.}
\end{equation}

Surprisingly, it turns out that matrix coefficients of general $K$-eigenfunctions
in $L^2(G)$ can be explicitly estimated in terms of the Harish-Chandra function:

\begin{Prop}
	\label{p:harish}
	For all $K$-eigenfunctions $\phi,\psi\in L^2(G)$,
	$$
	|\left<\lambda_G(g)\phi,\psi\right>|\le \Xi(g)\|\phi\|_2\|\psi\|_2.
	$$
\end{Prop}

The following argument is a version of Herz's majoration principle \cite{herz}.

\begin{proof}[Proof of Proposition \ref{p:harish}]
Replacing $\phi$ and $\psi$ by $|\phi|$ and $|\psi|$, we may assume without loss
of generality that $\phi,\psi\ge 0$, and the functions $\phi$ and $\psi$ are $K$-invariant.
Then because of the Cartan decomposition $G=KAK$, it is sufficient to prove the estimate when $g=a\in A$.
Using the decomposition of the invariant measure on $G$ given by \eqref{eq:iwa_measure},
we obtain that
\begin{align*}
\left<\lambda_G(a)\phi,\psi\right>= \int_G \phi(ga)\psi(g)\, dg
=\int_{U\times A\times K} \phi(ubka)\psi(ubk)\Delta(b)\, dudbdk.
\end{align*}
Then by the Cauchy--Schwarz inequality,
$$
\left<\lambda_G(a)\phi,\psi\right>\le \int_K 
\left(\int_{U\times A} \phi^2(ubka)\Delta(b)dudb\right)^{1/2}
\left(\int_{U\times A} \psi^2(ubk)\Delta(b)dudb\right)^{1/2}\, dk.
$$
Using that $\psi$ is $K$-invariant, we obtain that
\begin{align*}
\int_{U\times A} \psi^2(ubk)\Delta(b)dudb
=\int_{U\times A\times K} \psi^2(ubk)\Delta(b)dudbdk=\|\psi\|_2^2.
\end{align*}
To estimate the other term, we write
\begin{align*}
ubka=ub\cdot {\sf u}(ka){\sf a}(ka){\sf k}(ka)=ub{\sf u}(ka)b^{-1}\cdot b{\sf a}(ka)\cdot {\sf k}(ka).
\end{align*}
Since $\phi$ is $K$-invariant, using the invariance of the integrals, we deduce that
\begin{align*}
&\int_K 
\left(\int_{U\times A} \phi^2(ubka)\Delta(b)dudb\right)^{1/2}
 dk\\
 =&
\int_K 
\left(\int_{U\times A} \phi^2(ub{\sf u}(ka)b^{-1}\cdot b{\sf a}(ka))\Delta(b)dudb\right)^{1/2}
 dk\\
=&
  \int_K 
  \left(\int_{U\times A} \phi^2(u\cdot b)\Delta(b{\sf a}(ka)^{-1})dudb\right)^{1/2}
  dk\\
 =& \left(\int_K \Delta({\sf a}(ka))^{-1/2}dk\right)
\left(\int_{U\times A} \phi^2(ub)\Delta(b)dudb\right)^{1/2}.
\end{align*}
Finally, because of the $K$-invariance of $\phi$,
\begin{align*}
\int_{U\times A} \phi^2(ub)\Delta(b)dudb
=\int_{U\times A\times K} \phi^2(ubk)\Delta(b)dudbdk=\|\phi\|_2^2
\end{align*}
so that
\begin{align*}
\int_K 
\left(\int_{U\times A} \phi^2(ubka)\Delta(b)dudb\right)^{1/2}
dk=\Xi(a)\|\phi\|_2.
\end{align*}
This implies the required estimate.
\end{proof}	

Using Proposition \ref{p:harish}, we deduce our main result about 
representations of the group $L=\hbox{\rm SL}_2(\R)\ltimes \R^2$:

\begin{Theorem}\label{thy:semi}
	Let $\pi$ be a unitary representation of $L=\hbox{\rm SL}_2(\R)\ltimes \R^2$ on a Hilbert space $\mathcal{H}$
	such that there is no nonzero $\R^2$-invariant vectors.
	Then for all elements $g\in \hbox{\rm SL}_2(\R)$ and all $\hbox{\rm SO}(2)$-finite vectors $v,w\in \mathcal{H}$,
	$$
	|\left<\pi(g)v,w\right>|\le d_{\hbox{\tiny\rm SO}(2)}(v)^{1/2}d_{\hbox{\tiny\rm SO}(2)}(w)^{1/2} \|v\|\|w\|\, \Xi(g)^{1/2}.
	$$
\end{Theorem}

More general results giving quantitative bounds for representations
of semidirect products have been established by Oh \cite{oh2} and Wang \cite{wang}. 

\medskip

In relation to Proposition \ref{p:harish}, we mention that
Cowling, Haagerup, and Howe \cite{chh} discovered that the bound in Proposition \ref{p:harish}
holds more generally for any representation which is $L^{2+\epsilon}$-integrable
for all $\epsilon>0$:

\begin{theo}
	\label{th:chh}
	Let $G$ be a semisimple real algebraic group and $\pi$
	a unitary representation of $G$ on a Hilbert space $\mathcal{H}$
	which is $L^{2+\epsilon}$-integrable for all $\epsilon>0$.
	Then for all elements $g\in G$ and all $K$-finite vectors $v,w\in \mathcal{H}$,
	$$
	|\left<\pi(g)v,w\right>|\le d_K(v)^{1/2}d_K(w)^{1/2} \|v\|\|w\|\, \Xi(g).
	$$
\end{theo}
In view of Corollary \ref{c:integral}, Theorem \ref{th:chh}
applies to the setting of Theorem \ref{thy:semi} and implies a bound which
is essentially optimal in terms of the decay rate along $G$.

\begin{theo}\label{thy:semi1}
	Let $\pi$ be a unitary representation of $L=\hbox{\rm SL}_2(\R)\ltimes \R^2$ on a Hilbert space $\mathcal{H}$
	such that there is no nonzero $\R^2$-invariant vectors.
	Then for all elements $g\in \hbox{\rm SL}_2(\R)$ and all $\hbox{\rm SO}(2)$-finite vectors $v,w\in \mathcal{H}$,
	$$
	|\left<\pi(g)v,w\right>|\le d_{\hbox{\tiny\rm SO}(2)}(v)^{1/2}d_{\hbox{\tiny\rm SO}(2)}(w)^{1/2} \|v\|\|w\|\, \Xi(g).
	$$
\end{theo}

Here we only prove the weaker bound given by Theorem \ref{thy:semi}:
	
\begin{proof}[Proof of Theorem \ref{thy:semi}]
First, we consider the case when $v$ and $w$ are eigenvectors of $K_S=\hbox{SO}(2)$.
We recall that by Corollary \ref{c:embed}, the representation $(\pi\otimes \pi)|_S$,
where $S=\hbox{\rm SL}_2(\R)$,
embeds as a subrepresentation of $\bigotimes_{n\ge 1}\lambda_S$.
Hence, it follows from Proposition \ref{p:harish} that
for $g\in S$,
$$
|\left<\pi(g)v,w\right>|=
|\left<(\pi\otimes \pi)(g)(v\otimes v),w\otimes w\right>|^{1/2}\le  \|v\|\|w\|\,\Xi(g)^{1/2}.
$$
In general, we write $v$ and $w$ as $v=\sum_{i=1}^n v_i$ and $v=\sum_{j=1}^m w_j$,
where $v_i$'s and $w_j$'s are orthogonal $K_S$-eigenvectors. Then for $g\in S$,
\begin{align*}
|\left<\pi(g)v,w\right>|&\le \sum_{i=1}^n\sum_{j=1}^m |\left<\pi(g)v_i,w_j\right>|\le 
\left(\sum_{i=1}^n \|v_i\|\right) \left(\sum_{j=1}^m \|w_j\|\right) \Xi(g)^{1/2}\\
&\le n^{1/2} \left(\sum_{i=1}^n \|v_i\|^2\right)^{1/2}\, m^{1/2}\left(\sum_{j=1}^m \|w_j\|^2\right)^{1/2}\, \Xi(g)^{1/2}\\
&\le d_{K_S}(v)^{1/2}d_{K_S}(w)^{1/2} \|v\|\|w\|\, \Xi(g)^{1/2}.
\end{align*}
This proves the theorem.
\end{proof}

Now we can derive uniform bounds for matrix coefficients of higher-rank simple Lie groups and prove Theorem \ref{th:high-rank}:

\begin{proof}[Proof of Theorem \ref{th:high-rank}]
We give a proof of the theorem for 
$$
G=\hbox{SL}_d(\R)\supset K=\hbox{SO}(d).
$$	
Because of the Cartan decomposition 
$$
G=KAK\quad\hbox{where $A=\{(a_1,\ldots,a_d):\, a_1,\ldots, a_d>0,\,  a_1\cdots a_d=1\}$,}
$$
it is sufficient to prove this estimate when $g=a\in A$.

We consider the subgroup $L=S\ltimes \R^2$,
where $S=\hbox{SL}_2(\R)$, embedded into the top left corner of $G$. It follows from Theorem \ref{th:hm2} that for all $v,w\in \mathcal{H}$,
$$
\left<\pi(g)v,w\right>\to 0\quad \hbox{as $g\to \infty$ in $G$.}
$$
In particular, it follows that there is no nonzero $\R^2$-invariant vectors.
Hence, Theorem \ref{thy:semi} can be applied to the representation $\pi|_L$.
We write $a\in A$ as $a=a'a''$ with
\begin{align*}
a'&=\hbox{diag}\left((a_1/a_2)^{1/2}, (a_2/a_1)^{1/2}, 1,\ldots,1\right),\\
a''&=\hbox{diag}\left((a_1a_2)^{1/2}, (a_1a_2)^{1/2}, a_3,\ldots,a_d\right).
\end{align*}
We note that $a'\in A_S\subset L$ and $a''$ commutes with $S$.
In particular, it commutes with $K_S=\hbox{SO}(2)$.
This implies that the vector $\pi(a'')v$ is $K_S$-finite,
and 
$$
d_{K_S}(\pi(a'')v)\le d_{K_S}(v)\le d_{K}(v).
$$
Hence, we deduce from Theorem \ref{thy:semi} and \eqref{eq:xi_b} that
\begin{align*}
|\left<\pi(a)v,w\right>|&=|\left<\pi(a')\pi(a'')v,w\right>|
\le d_{K_S}(\pi(a'')v)^{1/2}d_{K_S}(w)^{1/2} \|v\|\|w\|\, \Xi(a')^{1/2}\\
&\ll_\epsilon d_{K}(v)^{1/2}d_{K}(w)^{1/2} \|v\|\|w\|\, \left(\frac{a_1}{a_2}\right)^{-1/4+\epsilon}
\end{align*}
for all $\epsilon>0$.
The same argument can be applied to other embeddings of $\hbox{SL}_2(\R)\ltimes \R^2$
into $\hbox{SL}_d(\R)$. This gives the bound
\begin{align*}
|\left<\pi(a)v,w\right>|
&\ll_\epsilon d_{K}(v)^{1/2}d_{K}(w)^{1/2} \|v\|\|w\|\, \left(\max_{i\ne j }\frac{a_i}{a_j}\right)^{-1/4+\epsilon}
\end{align*}
for all $\epsilon>0$, and proves the theorem.
\end{proof}

\medskip

It is useful for applications to have the estimate as Theorem \ref{th:high-rank}
in terms of H\"older norms or Sobolev norms of smooth vectors,
as in the works of Moore \cite{m}, Ratner \cite{rat}, and 
Katok and Spatzier \cite{KS}.
Given a unitary represenation $\pi$ of $G$
on a Hilbert space $\mathcal{H}$, one can also define an action of the Lie algebra
$\hbox{Lie}(G)$ on a dense subspace $\mathcal{V}$ of $\mathcal{H}$ that satisfies
$$
\pi(\mathcal{X})v=\frac{d}{dt}\pi(\exp(t\mathcal{X}))v|_{t=0}\quad \hbox{for $\mathcal{X}\in \hbox{Lie}(G)$ and $v\in\mathcal{V}$.}
$$
We fix an (ordered) basis $(\mathcal{X}_1,\ldots,\mathcal{X}_n)$ of the Lie algebra
$\hbox{Lie}(G)$. Then the Sobolev norm of order $\ell$ is defined as 
\begin{equation}
\label{eq:sobolev}
S_\ell(v)^2=\sum_{(i_1,\ldots,i_\ell)} \left\|\pi(\mathcal{X}_{i_1})\ldots \pi(\mathcal{X}_{i_\ell})v\right\|^2
\end{equation}
With this notation, we prove:

\begin{Theorem}
	\label{th:high-rank2}
	Let $G$ be a (noncompact) connected   simple higher-rank matrix Lie group with finite centre
	and $K$ a maximal compact subgroup of $G$.
	Then there exist $c,\delta,\ell>0$ such that for any unitary representation $\pi$ of 
	$G$ on a Hilbert space $\mathcal{H}$ without nonzero $G$-invariant vectors,
	$$
	|\left<\pi(g)v,w\right>|\le c\,  S_\ell(v) S_\ell(w)\, \|g\|^{-\delta}\quad\hbox{for all $g\in G$ and all $v,w\in \mathcal{V}$}.
	$$
\end{Theorem}

\begin{proof}
The proof will require several more advanced facts about representations of semisimple groups.
We indicate how to complete the proof using these facts when $\pi$ is irreducible.
Then the general case will follow by using the integral decomposition.

We decompose $\mathcal{H}$ as a direct sum of irreducible representations of a maximal compact subgroup of $K$. This gives the decomposition
\begin{equation}
\label{eq:decomp}
\mathcal{H}=\bigoplus_{\sigma\in \hat K} \mathcal{H}_\sigma,
\end{equation}
where $\hat K$ denotes the unitary dual of $K$, and $\mathcal{H}_\sigma$ is the direct sum of 
the irreducible components isomorphic to $\sigma$. There is the Casimir operator
$\mathcal{C}$ of $K$ which 
is the second order differential operator commuting with the action of $K$.
It leaves each of the subspaces $\mathcal{H}_\sigma$ invariant.
Moreover, one can deduce from the Schur Lemma that 
$$
\mathcal{C}|_{H_\sigma}=\lambda_\sigma\, \hbox{id}_{\mathcal{H}_\sigma}
$$
for some $\lambda_\sigma>0$.
The eigenvalues $\lambda_\sigma$ and the dimensions $\dim(\sigma)$ are computed 
explicitly in the Representation Theory of compact groups, and one can verify that 
\begin{align*}
\dim(\sigma)\le \lambda_\sigma^{c_1}\quad\hbox{and}\quad
\sum_{\sigma\in\hat K} \dim(\sigma)^{-c_2}<\infty
\end{align*}
for some $c_1,c_2>0$.
We shall also use a result of Harish-Chandra regarding admissibility of
 irreducible unitary representation (see, for instance, \cite{war1}) that gives the bound
$$
\dim (H_\sigma)\le \dim(\sigma)^2.
$$

Now utilising these estimates, we proceed with the proof of the theorem.
We decompose the vectors with respect to the decomposition
\eqref{eq:decomp} and deduce from Theorem \ref{th:high-rank} that
\begin{align*}
|\left<\pi(g)v,w\right>|&\le \sum_{\sigma,\tau\in \hat K}|\left<\pi(g)v_\sigma,w_\tau\right>|\\
&\ll \left(\sum_{\sigma\in \hat K} \dim(H_\sigma)^{1/2}\|v_\sigma\|\right)
\left(\sum_{\tau\in \hat K} \dim(H_\tau)^{1/2}\|w_\tau\|\right)\, \|g\|^{-\delta}.
\end{align*}
Since the above sums can be estimates as 
\begin{align*}
\sum_{\sigma\in \hat K} \dim(H_\sigma)^{1/2}\|v_\sigma\|
&\le \sum_{\sigma\in \hat K} \dim(\sigma) \lambda_\sigma^{-s}\left\|\pi(\mathcal{C})^s v_\sigma\right\|\\
&\le \sum_{\sigma\in \hat K} \dim(\sigma)^{1-s/c_1} \left\|\pi(\mathcal{C})^s v_\sigma\right\|
\\
&\le \left(\sum_{\sigma\in\hat K} \dim(\sigma)^{-2(s/c_1-1)}\right)^{1/2}
\left(\sum_{\sigma\in\hat K}  \left\| \pi(\mathcal{C})^sv_\sigma\right\|^2\right)^{1/2} \\
&\ll  \left\|\pi(\mathcal{C})^s v\right\|
\end{align*}
for sufficiently large $s$, this implies the theorem. 
\end{proof}

We note that if the assumption that $G$ is of higher rank 
is removed, the statements of Theorems \ref{th:high-rank} and \ref{th:high-rank2} are not true.
Although we know from Theorem \ref{th:hm2} that 
$$
\left<\pi(g)v,w\right>\to 0 \quad \hbox{as $g\to \infty$ in $G$,}
$$
there are unitary representations without invariant vectors 
whose matrix coefficients do not possess explicit estimates.
For example, in the case of $\hbox{SL}_2(\R)$, the complementary series
representations provide examples with arbitrary slow decay rate.
Nonetheless, it is known that for every nontrivial irreducible representation $\pi$, there exists $c(\pi),\delta(\pi)>0$ such that 
$$
|\left<\pi(g)v,w\right>|\le c(\pi)\,  S_\ell(v) S_\ell(w)\, \|g\|^{-\delta(\pi)}.
$$
Moreover, this bound also holds for any unitary representation $\pi$
which is isolated from the trivial representation in the sense of the Fell topology.

\medskip

Theorem \ref{th:high-rank2} can be applied to finite-volume 
homogeneous spaces $X$ of $G$. Indeed, a simple Fubini-type argument 
implies that the corresponding unitary representation of $G$
on $L^2(X)$ has no nonconstant invariant vectors so that
the bound of Theorem \ref{th:high-rank2} can be applied to all 
function in $L_0^2(X)$, which denotes the subspace of functions with zero integral.  Even when $G$ has rank one, one can show that the unitary 
representation of $G$ on $L^2_0(X)$ is isolated from the trivial representation
(see, for instance, \cite[Lemma~3]{bekka})
so that the quantitative bounds on correlations hold in this case as well.

\begin{theo}
	\label{th:high-rank3}
	Let $G$ be a (noncompact) connected simple matrix Lie group with finite centre and $(X,\mu)$ is a probability homogeneous space of $G$.
	Then there exist $\delta,\ell>0$ such that for 
	all elements $g\in G$ and all functions $\phi,\psi\in C_c^\infty(X)$,
	$$
	\int_X \phi(g^{-1}x)\psi(x)\,d\mu(x) =\left(\int_X\phi\, d\mu\right)\left(\int_X\psi\, d\mu\right)+
	O\left( S_\ell(\phi) S_\ell(\psi)\, \|g\|^{-\delta}\right).
	$$
\end{theo}

Maucourant \cite{mauc}
used estimates on the correlations from Theorem \ref{th:high-rank3}
to prove a version of Theorem \ref{th:count} that gives an asymptotic formula for the
number of lattice points with an error term.

\medskip

In conclusion, 
we note that in Theorems \ref{th:high-rank}, \ref{th:high-rank2} and \ref{th:high-rank3},
the matrix norm $\|\cdot\|$
can be estimated in terms of a (left) invariant Riemannian metric $d$ on $G$ as 
$$
e^{c_1 d(g,e)}\le  \|g\|\le e^{c_2 d(g,e)}\quad\hbox{for all $g\in G$.}
$$
In particular,  in Theorem \ref{th:high-rank3}, this gives the error term 
$$
O\left(  S_\ell(\phi) S_\ell(\psi)\, e^{-\delta'd(g,e)}\right)
$$
for some $\delta'>0$.

\section{Bounds on higher-order correlations}\label{sec:cor_gen}

Building on the results from Section \ref{sec:cor_2}, we intend to establish 
quantitative estimates on correlations of arbitrary order. We follow the argument
of Bj\"orklund, Einsiedler, Gorodnik \cite{beg}.
Throughout this section, $G$ denotes a (noncompact) connected  simple matrix Lie group
with finite center. We consider a measure preserving action of $G$ on 
a standard probability space $(X,\mu)$. To simplify notation,
we set
$$
(g\cdot \phi)(x)=\phi(g^{-1}x)\quad\hbox{for $g\in G$ and $\phi\in L^\infty(X)$.}
$$
Our goal is to estimate the correlations
$$
\mu((g_1\cdot \phi_1)\cdots (g_r\cdot \phi_k))=\int_X \phi_1(g_1^{-1}x)\cdots \phi_r(g_k^{-1}x)\, d\mu(x)
$$
for suitable functions $\phi_1,\ldots,\phi_r$ on $X$.
We shall assume that we know how to estimate correlations of order two.
Namely, we assume that there exist a subalgebra $\mathcal{A}$ of $L^\infty(X)$
and  $\delta>0$ such that for all functions $\phi_1,\phi_2\in\mathcal{A}$
 and all $g\in G$,
\begin{equation}
\label{eq:m}
\mu((g\cdot \phi_1)\, \phi_2)=\mu(\phi_1)\mu(\phi_2)+O\big(S_\ell(\phi_1)S_\ell(\phi_2)\, \|g\|^{-\delta}\big),
\end{equation}
where $S_\ell$ denotes a norm on the algebra $\mathcal{A}$.
The precise definition of the family of norms 
$$
S_1\le S_2\le \cdots \le S_\ell\le \cdots
$$
will not be important for our arguments. We shall only use that these norms
 satisfy the following properties:
 \begin{enumerate}
 \item[(${\rm N}_1$)] there exists $\ell_1$ such that
$$
 \|\phi\|_{L^\infty}\ll S_{\ell_1}(\phi),
$$
\item[(${\rm N}_2$)] there exists $\ell_2$ such that 
$$
\|g\cdot \phi-\phi\|_{L^\infty}\ll \|g-e\|\, S_{\ell_2}(\phi)\quad\hbox{for all $g\in G$,}
$$
\item[(${\rm N}_3$)] for all $\ell$, there exists $\sigma_\ell>0$ such that
$$
S_\ell(g\cdot \phi)\ll_\ell \|g\|^{\sigma_\ell} \, S_\ell(\phi)\quad\hbox{for all $g\in G$,}
$$
\item[(${\rm N}_4$)] for every $\ell$, there exists $\ell'$ such that
$$
S_\ell(\phi_1\phi_2)\ll_\ell S_{\ell'}(\phi_1) S_{\ell'}(\phi_2).
$$
\end{enumerate}

\medskip

For instance, when $X=L/\Gamma$ where $L$ is a connected Lie group and $\Gamma$ 
a discrete cocompact subgroup,  it follows from a version of the Sobolev embedding theorem
that the Sobolev norms defined in \eqref{eq:sobolev} satisfy these properties.
More generally, when $\Gamma$ is discrete subgroup of finite covolume,
one can also introduce a family of norms majorating the usual Sobolev norms
satisfying these properties (see, for instance, \cite[\S3.7]{emv}). In particular, it follows from Section \ref{sec:cor_2}
that the bound \eqref{eq:m} holds in this setting.

\medskip

The main result of this section is the following:

\begin{Theorem}	\label{th:cor_high}
For every $r\ge 2$, there exist $\delta_r,\ell_r>0$ such that
for all elements $g_1,\ldots,g_r\in G$ and all functions $\phi_1,\ldots,\phi_r\in \mathcal{A}$,
\begin{align*}
\mu((g_1\cdot \phi_1)\cdots (g_r\cdot \phi_r))=\, & \mu(\phi_1)\cdots \mu(\phi_r)\\
&+O_r\big( S_{\ell_r}(\phi_1)\cdots S_{\ell_r}(\phi_r)\, N(g_1,\ldots,g_r)^{-\delta_r}\big),
\end{align*}
where
$$
N(g_1,\ldots,g_r)=\min_{i\ne j} \|g_i^{-1}g_j\|.
$$
\end{Theorem}

We first explain the strategy of the proof of Theorem \ref{th:cor_high}.
It will be convenient to consider the correlation of order $r$ as a measure on the product
space $X^r$: we introduce a measure $\eta=\eta_{g_1,\ldots,g_r}$ on $X^r$
defined by
$$
\eta(\phi)=\int_X\phi(g_1^{-1}x,\ldots, g_r^{-1}x)\, d\mu(x)\quad\hbox{for $\phi\in L^\infty(X^r)$.}
$$
Theorem \ref{th:cor_high} amounts to showing that 
the measure $\eta$ is ``approximately'' equal  to the product 
measure $\mu^r$ on $X^r$.
Our argument will proceed by induction on the number of factors. 
Let us take a nontrivial partition $\{1,\ldots,r\}=I\sqcup J$
which defines the projection maps $X^r\to X^I$ and $X^r\to X^J$.
We obtain the measures $\eta_I$ and $\eta_J$ on $X^I$ and $X^J$ respectively
that are the projections of the measure $\eta$. Formally, these measure are defined as
\begin{align*}
\eta_I(\phi)&=\eta(\phi\otimes 1)\quad\hbox{for $\phi\in L^\infty(X^I)$,}\\
\eta_J(\psi)&=\eta(1\otimes \psi)\quad\hbox{for $\psi\in L^\infty(X^J)$.}
\end{align*}
Ultimately, our proof will involve comparing the diagrams:

$$
\xymatrixcolsep{1.3pc} \xymatrixrowsep{1.2pc}\
\xymatrix{
&& X^r  \ar@/_1pc/[lldd] \ar@/^1pc/[rrdd]    \\	
&& \eta \ar@/_/[ld] \ar@/^/[rd] \ar@{.>}[u] &&\\
X^I & \eta_I \ar@{.>}[l] && \eta_J \ar@{.>}[r] & X^J
}\;\;
\xymatrix{
	&& X^r  \ar@/_1pc/[lldd] \ar@/^1pc/[rrdd]   && \\	
	&& m^r \ar@/_/[ld] \ar@/^/[rd] \ar@{.>}[u] &&\\
	X^I & m^I \ar@{.>}[l] && m^J \ar@{.>}[r] & X^J
}
$$

We may assume by induction that 
$$
\eta_I\approx \mu^I\quad\hbox{ and } \quad \eta_J\approx \mu^J
$$
and need to show that 
\begin{equation}
\label{eq:approx}
\eta\approx \mu^r= \mu^I\otimes \mu^J.
\end{equation}
To establish this estimate, we use that the measure $\eta$ is invariant under the subgroup
$$
D=\{(g_1^{-1}hg_1,\ldots, g_r^{-1}hg_r): \, h\in G\}.
$$
We take a one-parameter subgroup 
$$
h(t)=(\exp(tZ_1),\ldots,\exp(tZ_r))
$$
in $D$ which also can be written as
$$
h(t)=(h_I(t),h_J(t))
$$
for one-parameter subgroups acting on $X^I$ and $X^J$.
We note that the measure $\eta_I$ is $h_I(t)$-invariant, 
and the measure $\eta_J$ is $h_J(t)$-invariant.
We consider the averaging operator
\begin{equation}
\label{eq:P_T}
P_T:L^\infty(X^I)\to L^\infty(X^I):\, \phi \mapsto \frac{1}{T}\int_0^T \phi(h_I(t)x)\, dt
\end{equation}
that preserves the measure $\eta_I$. Given functions $\phi_1,\ldots,\phi_r\in L^\infty(X)$,
we write
$$
\phi_I=\otimes_{i\in I} \phi_i\quad\hbox{and} \quad \phi_J=\otimes_{i\in J} \phi_i.
$$
To simplify notation, we write $S_\ell(\phi_I)=\prod_{i\in I} S_\ell(\phi_i)$ below.
We establish \eqref{eq:approx} using the following key estimate
\begin{align}
|\eta(\phi_I\otimes \phi_J)-\mu^r(\phi_1\otimes \cdots \phi_r)|=&\;
|\eta(\phi_I\otimes \phi_J)-\mu^I(\phi_I)\mu^J(\phi_J)| \label{eq:key}\\
\le &\;\; |\eta(\phi_I\otimes \phi_J)-\eta(P_T\phi_I\otimes \phi_J)| \tag{I}\\
&\;+|\eta(P_T\phi_I\otimes \phi_J)-\eta_I(\phi_I)\eta_J(\phi_J)| \tag{II}\\
&\;+|\eta_I(\phi_I)\eta_J(\phi_J)-\mu_I(\phi_I)\mu_J(\phi_J)|. \tag{III}
\end{align}
We will estimate the terms (I), (II), and (III) separately for a carefully chosen partition $\{I,J\}$ and
a carefully chosen one-parameter subgroup $h(t)$. To simplify notation,
we shall assume that for all $i=1,\ldots,r$,
$$
S_{\ell'}(\phi_i)\le 1
$$
for a fixed sufficiently large $\ell'$. In particular,
it also follows from properties (${\rm N}_1$) and (${\rm N}_2$) of the norms that 
for all $i=1,\ldots,r$ and $g\in G$,
\begin{equation}
\label{eq:bbound}
\|\phi_i\|_{L^\infty}\ll 1\quad\hbox{and}\quad \|g\cdot \phi_i-\phi_i\|_{L^\infty}\ll \|g-e\|.
\end{equation}

It will be convenient to replace the matrix norm by a different norm defined
in terms of the adjont representation $\hbox{Ad}:G\mapsto \hbox{GL}(\hbox{Lie}(G))$ 
with 
$$
\hbox{Ad}(g):X\mapsto gX g^{-1}\quad\hbox{for $X\in \hbox{Lie}(G)$.}
$$
We fix a norm on the Lie algebra $\hbox{Lie}(G)$ and set
$$
\|g\|=\max\left\{\|\hbox{Ad}(g)Z\|: \|Z\|=1\right\}.
$$
It is not hard to check that for any $g\in G$, 
\begin{align*}
&\|g\|\ge 1\quad\hbox{and}\quad
&\|g\|=\|\hbox{Ad}(g)Z\|\quad\hbox{for some nilpotent $Z$ with $\|Z\|=1$.}
\end{align*}

Now we describe the choice of the one-parameter $h(t)$ subgroup that we use.
Let
$$
Q=\max_{i\ne j} \|g_i^{-1}g_j\|\quad\hbox{and}\quad q=\min_{i\ne j} \|g_i^{-1}g_j\|\ge 1
$$
We take $i_1\ne i_s$ such that 
$$
Q=\|g_{i_1}^{-1}g_{i_s}\|=\|\hbox{Ad}(g_{i_1}^{-1}g_{i_s})Z\|
$$
for some nilpotent $Z$ with $\|Z\|=1$. Then
$$
\|\hbox{Ad}(g_{i_1}^{-1}g_{i_s})Z\|\ge \|\hbox{Ad}(g_{i}^{-1}g_{j})Z\|\quad\hbox{for all $i\ne j$.}
$$
For a suitable choice of indices, we obtain that
$$
\|\hbox{Ad}(g_{i_1}^{-1}g_{i_s})Z\|\ge \|\hbox{Ad}(g_{i_2}^{-1}g_{i_s})Z\|\ge \cdots \ge \|\hbox{Ad}(g_{i_r}^{-1}g_{i_s})Z\|.
$$
In fact, after relabelling, we may assume that
$$
\|\hbox{Ad}(g_{1}^{-1}g_{s})Z\|\ge \|\hbox{Ad}(g_{2}^{-1}g_{s})Z\|\ge \cdots \ge \|\hbox{Ad}(g_{r}^{-1}g_{s})Z\|.
$$
We note that 
$$
\|\hbox{Ad}(g_{r}^{-1}g_{s})Z\|\le \|\hbox{Ad}(g_{s}^{-1}g_{s})Z\|=1.
$$
We set
$$
Z_j=\frac{\hbox{Ad}(g_{j}^{-1}g_{s})Z}{\|\hbox{Ad}(g_{1}^{-1}g_{s})Z\|}\quad\hbox{and}\quad
w_j=\|Z_j\|.
$$
Then
\begin{equation}
\label{eq:w}
1=w_1\ge w_2\ge \cdots \ge w_r\quad\hbox{and}\quad w_r\le Q^{-1}\le q^{-1}
\end{equation}
We take 
$$
I=\{1,\ldots,p\}\quad\hbox{and}\quad I=\{p+1,\ldots,r\},
$$
where the index $p$ will be specified later.

We observe that with these choices, the one-parameter subgroups $h_I(t)$ and $h_J(t)$
satisfy the following properties with some exponents $a,b>0$,
\begin{enumerate}
\item[(a)] $\|h_J(t)\cdot \phi_J -\phi_J\|_{L^\infty} \ll w_{p+1}|t|$,

\item[(b)] $S_\ell (h_I(t)\cdot \phi_I)\ll \max(1,|t|)^a$,

\item[(c)] $|\mu^I((h_I(t)\cdot \phi_I) \phi_I)-\mu^I(\phi_I)^2|\ll \max(1,w_p|t|)^{-b}$.
\end{enumerate}
Indeed, (a) can be deduced from the property (${\rm N}_2$) of the Sobolev norms,
(b) --- from the property (${\rm N}_3$), and (c) --- from the bound \eqref{eq:m}
on correlations of order two.

\medskip

Now we proceed to estimate \eqref{eq:key}. 
Our argument proceeds by induction on $r$, and we suppose that 
we have established existence of $E=E(g_1,\ldots,g_r)$ such that for all 
proper subsets $L$ of $\{1,\ldots,r\}$ and functions $\psi_1,\ldots,\psi_r\in \mathcal{A}$,
\begin{equation}
\label{eq:induction}
|\eta_L(\psi_L)-\mu^L(\psi_L)|\le E\, S_\ell(\psi_L).
\end{equation}
We estimate each of the terms (I), (II), (III) appearing in \eqref{eq:key}.
We note that it follows immediately from the assumption \eqref{eq:induction} and \eqref{eq:bbound}
that 
\begin{equation}
\label{eq:III}
|\eta_I(\phi_I)\eta_J(\phi_J)-\mu^I(\phi_I)\mu^J(\phi_J)|\ll E.
\end{equation}
This provides an estimate for the term (III).

\medskip

To estimate the term (I), we observe that
$$
\eta(P_T\phi_I\otimes \phi_J)=\eta\left(\frac{1}{T}\int_0^T (h_I(t)\cdot \phi_I)\otimes \phi_J\, dt \right).
$$
Using that the measure $\eta$ is invariant under $h(t)=(h_I(t),h_J(t))$,
we obtain that 
$$
\eta(\phi_I\otimes \phi_J)=\eta\left(\frac{1}{T}\int_0^T (h_I(t)\cdot \phi_I)\otimes (h_J(t)\cdot\phi_J)\, dt \right).
$$
Hence, the term (I) can be estimated as
\begin{align} \label{eq:I}
&|\eta(\phi_I\otimes \phi_J)-\eta(P_T\phi_I\otimes \phi_J)|\\
\nonumber \le&\,
\eta\left(\frac{1}{T}\int_0^T \big| (h_I(t)\cdot \phi_I)\otimes (h_J(t)\cdot\phi_J)
-(h_I(t)\cdot \phi_I)\otimes \phi_J\big|\, dt \right)\\
\nonumber \le&\, \frac{1}{T}\int_0^T \big\|(h_I(t)\cdot \phi_I)\otimes (h_J(t)\cdot\phi_J -\phi_J)\big\|_{L^\infty}\, dt\\
\nonumber \le & \, \|\phi_I\|_{L^\infty}\cdot \max_{0\le t\le T} \|h_J(t)\cdot\phi_J-\phi_J\|_{L^\infty}
\ll w_{p+1} T,
\end{align}
where we used \eqref{eq:bbound} and (a).

\medskip

To estimate the term (II), we use that
$$
\eta_I(\phi_I)\eta_J(\phi_J)=\eta_I(\phi_I)\eta(1\otimes \phi_J)=\eta(\eta_I(\phi_I)\otimes \phi_J).
$$
We first show that the term (II) can be estimated  in terms of the quantity
$$
D_T(\eta_I)=\eta_I\left( |P_T\phi_I-\eta_I(\phi_I)|^2\right)^{1/2}.
$$
Indeed, we obtain that 
\begin{align*}
|\eta(P_T\phi_I\otimes \phi_J)-\eta_I(\phi_I)\eta_J(\phi_J)|
&=|\eta((P_T\phi_I-\eta_I(\phi_I))\otimes \phi_J)|\\
&\le \eta\left(|P_T\phi_I-\eta_I(\phi_I)|\otimes |\phi_J|\right)\\
&\le \eta\left(|P_T\phi_I-\eta_I(\phi_I)|\right)\|\phi_J\|_{L^\infty}\\
&\le D_T(\eta_I)
\end{align*}
by \eqref{eq:bbound} and the Cauchy--Schwarz inequality.
To deal with $D_T(\eta_I)$, we use that it can be approximated by
$$
D_T(\mu^I)=\mu^I\left( |P_T\phi_I-\mu^I(\phi_I)|^2\right)^{1/2}.
$$
The corresponding estimate is given by:

\begin{Lemma}\label{l:D}
	$|D_T(\eta_I)-D_T(\mu^I)|\ll T^{a/2} E^{1/2}$.
\end{Lemma}

\begin{proof}
Using the inequality $|x-y|\le \sqrt{|x^2-y^2|}$ with $x,y\ge 0$, we obtain that
$$
|D_T(\eta_I)-D_T(\mu^I)|\le \sqrt{|D_T(\eta_I)^2-D_T(\mu^I)^2|}.
$$
Expanding the averaging operator \eqref{eq:P_T} and changing the order of integration, we deduce that
\begin{align*}
D_T(\eta_I)^2&=\int_{X^I} |P_T\phi_I-\eta_I(\phi_I)|^2\, d\eta_I\\
&=\frac{1}{T^2}\int_0^T\int_0^T  \big( \eta_I ( (h_I(s-t)\cdot\phi_I) \phi_I)-\eta_I(\phi_I)^2  \big)\, dsdt.
\end{align*}
Similarly,
\begin{align*}
D_T(\mu^I)^2
&=\frac{1}{T^2}\int_0^T\int_0^T  \big( \mu^I ( (h_I(s-t)\cdot\phi_I) \phi_I)-\mu^I(\phi_I)^2  \big)\, dsdt.
\end{align*}
Hence,
\begin{align*}
&|D_T(\eta_I)^2-D_T(\mu^I)^2|\\
\le&\, \frac{1}{T^2}\int_0^T\int_0^T \Big( 
\big|\eta_I ( (h_I(s-t)\cdot\phi_I) \phi_I)-\mu^I ( (h_I(s-t)\cdot\phi_I) \phi_I) \big|
\\
&\quad\quad\quad\quad\quad\quad+
\big|\eta_I(\phi_I)^2-\mu^I(\phi_I)^2\big|
\Big)\, dsdt.
\end{align*}
The first term inside the integral is estimated using \eqref{eq:induction} as
\begin{align*}
&\ll E\, S_\ell \big((h_I(s-t)\cdot \phi_I)\phi_I\big)\ll E\, 
S_{\ell'} ((h_I(s-t)\cdot \phi_I) S_{\ell'}(\phi_I)\\
&\ll E\, \max(1,|s-t|)^a,
\end{align*}
where we used (${\rm N}_4$), (b), and \eqref{eq:bbound}.
The second term inside the integral is estimated using \eqref{eq:induction} and \eqref{eq:bbound} as
\begin{align*}
=|\eta_I(\phi_I)-\mu^I(\phi_I)|\cdot |\eta_I(\phi_I)+\mu^I(\phi_I)|\le E\cdot 2\|\phi_I\|_{L^\infty}\ll E.
\end{align*}
Finally, the lemma follows from the bound
$$
\frac{1}{T^2}\int_0^T\int_0^T  \max(1,|s-t|)^a\, dsdt\ll T^a.
$$
\end{proof}

It follows from (c) that 
\begin{align*}
D_T(\mu^I)^2
&=\frac{1}{T^2}\int_0^T\int_0^T  \big( \mu^I ( (h_I(s-t)\cdot\phi_I) \phi_I)-\mu^I(\phi_I)^2  \big)\, dsdt\\
&\ll 
\frac{1}{T^2}\int_0^T\int_0^T  \max(1,w_p|s-t|)^{-b}\, dsdt \ll (w_p T)^{-b}.
\end{align*}
Hence, we conclude from Lemma \ref{l:D} that the term (II) can be estimated as
\begin{equation}
\label{eq:II}
|\eta(P_T\phi_I\otimes \phi_J)-\eta_I(\phi_I)\eta_J(\phi_J)|\ll \max \big( T^{a/2}E^{1/2}, (w_pT)^{-b/2}\big).
\end{equation}

\medskip

Now combining the bounds \eqref{eq:I}, \eqref{eq:II}, and \eqref{eq:III}, 
we deduce from \eqref{eq:key} that for all $T\ge 1$,
$$
|\eta(\phi_I\otimes \phi_J)-\mu^I(\phi_I)\mu^J(\phi_J)|\ll 
\max \big(w_{p+1} T, T^{a/2}E^{1/2}, (w_pT)^{-b/2}\big).
$$
This estimate will be used to complete the proof of Theorem \ref{th:cor_high} by induction on $r$.
We suppose that \eqref{eq:induction} holds with $E=q^{-\tau}$ for some $\tau>0$.
We have to pick the index $p$ and the parameter $T$ to minimise
\begin{equation}
\label{eq:max}
\max \big(w_{p+1} T, T^{a/2}q^{-\tau/2}, (w_pT)^{-b/2}\big).
\end{equation}
We seek a bound which is uniform on the parameters $w_1,\ldots,w_r$ satisfying \eqref{eq:w}.
We take $\theta>0$ with $\theta<(r-1)^{-1}$.
Then since $w_r\le q^{-1}$, all the $r$ points 
$$
1,q^{-\theta}\ldots, q^{-(r-1)\theta}
$$
are contained in the union of $r-1$ intervals
$$
[w_r,w_{r-1}],\ldots, [w_2,w_1].
$$
Hence, by the Pigeonhole Principle, there exist $p$ and $i$ such that
$$
w_{p+1}\le q^{-(i+1)\theta} < q^{-i\theta}\le w_p.
$$
Taking $T=q^{(i+1/2)\theta}$, we obtain that \eqref{eq:max} is 
estimated by $q^{-\tau'}$ with $\tau'>0$. 
This completes the proof of Theorem \ref{th:cor_high}.

\section{Application: existence of configurations}\label{sec:conf}

Analysis of higher-order correlation can be used to study existence of 
combinatorial configurations. Perhaps, the most striking example
of this is the Szemer\'edi theorem \cite{sz} which states that any 
subset of integers of positive upper density contains arbitrary long 
arithmetic progressions. Furstenberg \cite{furst} discovered that this problem
can be modelled using dynamical systems. His approach is based on the ``Furstenberg Correspondence Principle''
which associates to a subset of positive density in $\mathbb{Z}$ a shift-invariant measure
on the space $\{0,1\}^\Z$. The crux of Furstenberg's proof \cite{furst} of 
the Szemer\'edi theorem is the following result which implies
nonvanishing of higher-order correlations.

\begin{theo}\label{th:furst}
Let $T:X\to X$ be a measure-preserving transformation of a probability space $(X,\mu)$.
Then for every nonnegative $\phi\in L^\infty(X)$ which is not zero almost everywhere,
$$
\liminf_{N\to\infty}\frac{1}{N}\sum_{i=0}^{N-1} \int_X \phi(x)\phi(T^{i}x)\cdots \phi (T^{(r-1)i}x)\, d\mu(x)>0.
$$
\end{theo}

This result was generalised by Furstenberg and Katznelson \cite{fk} to systems
of commuting transformations which allowed to prove the following generalisation 
of the Szemer\'edi theorem.

\begin{theo}\label{th:sz}
Let $\Omega$ be a subset of $\Z^d$ of positive upper density. Then
for any finite subset $F$ of $\Z^d$, there exist $a\in \mathbb{Z}^d$ and $t\in \N$
such that
$$
a+tF\subset \Omega.
$$
\end{theo}

We recall that a set $\Omega$ has {\it positive upper density} if there exists
a sequence of boxes $B_n$ with lengths of all sides going to infinity such that
$$
\limsup_{n\to \infty} \frac{|\Omega\cap B_n|}{|B_n|}>0.
$$

Existence of configurations in subsets of the Euclidean space $\R^d$ has been
also extensively studied. The following results was proved  by Furstenberg, Katznelson, and Weiss \cite{fkw} for $d=2$ using ergodic-theoretic
techniques and by Bourgain \cite{bur} in general using Fourier analysis.

\begin{theo}\label{th:conf1}
Let $\Omega$ be a subset of positive density in $\R^d$, and 	$F=\{0,x_1,\ldots,x_{d-1}\}$ is a subset of points in $\R^d$ in general position. Then there exists $t_0$ such that for 
every $t\ge t_0$, the set $\Omega$ contains an isometric copy of $tF$.
\end{theo}

It was shown by Bourgain \cite{bur} and  Graham \cite{gr} that an analogue of
this theorem fails for general configurations. 
Nonetheless, one may ask whether the set $\Omega$ contains approximate
configurations. This was settled by 
Furstenberg, Katznelson, and Weiss \cite{fkw}
by configurations of three points and by Ziegler \cite{zie} in general:

\begin{theo}\label{th:conf2}
	Let $\Omega$ be a subset of positive density in $\R^d$ and 
	$x_1,\ldots,x_{r-1}\in\R^d$. Then there exists $t_0$ such that
	for every $t\ge t_0$ and $\epsilon>0$, one can find $(y_0,y_1,\ldots,y_{r-1})\subset\Omega^{r}$ and an isometry of $I$ of $\R^d$ such that
	$$
	d(0,I(y_0))<\epsilon\quad\hbox{and}\quad d(tx_i,I(y_i))<\epsilon\quad\hbox{for $i=1,\ldots,r-1$.}
	$$
\end{theo}

The proof of Theorem \ref{th:conf2} requires more detailed analysis of the averages
of correlations 
\begin{equation}
\label{eq:limit}
\frac{1}{N}\sum_{i=0}^{N-1} \int_X \phi_0(x)\phi_1(T^{i}x)\cdots \phi_{r-1} (T^{(r-1)i}x)\, d\mu(x)
\end{equation}
for $\phi_0,\ldots,\phi_{r-1}\in L^\infty(X)$.
While the case when $r=3$ can be reduced to investigating translations on
compact abelian groups. The general case have presented a significant challenge
that was solved in the ground-breaking works of Host and Kra \cite{hk}, and Ziegler \cite{zie_factor}.
These works developed a comprehensive method that allowed to understand
limits in $L^2(X)$ of the averages \eqref{eq:limit}. It turns out that this reduces
to analysing 
this limit for the so-called characteristic factors which are shown to be
inverse limits of dynamical systems which are translations on nilmanifolds.
Thus, remarkably to investigate the general limits of the averages \eqref{eq:limit}
it suffices to deal with these limits for nilsystems. 
We also mention that 
Leibman \cite{lei} and Ziegler \cite{zie_0} established existence of the limit
of \eqref{eq:limit} for translations on nilmanifolds.

\medskip

More generally, let us consider a locally compact group $G$ equipped with a left-invariant metric. Given a ``large'' subset of $G$, we would like to show that 
it approximately contains an isometric copy of a given configuration $(g_1,\ldots,g_r)\in G^r$.
It is not clear what a natural notion of largeness in $G$ is, especially when the group $G$
is not amenable. In any case, one definitely views a subgroup $\Gamma$ in $G$ with finite covolume as being ``large''. We will be interested in investigated
 how rich the set of  configurations $(\gamma_1,\ldots, \gamma_r)\in \Gamma^r$ is.  In particular, one may wonder whether general configurations
 $(g_1,\ldots,g_r)\in G$ can be approximated by isometric copies of 
 the configurations  $(\gamma_1,\ldots, \gamma_r)\in \Gamma^r$ (see Figure \ref{f:conf}),
 namely, whether for every $\epsilon>0$, there exists an isometry $I:G\to G$
 such that 
 \begin{align}\label{eq:conf}
 d(g_i,I(\gamma_i))<\epsilon\quad\hbox{for $i=1,\ldots,r$}.
 \end{align}

\begin{figure}[h]
	\includegraphics[width=0.7\linewidth]{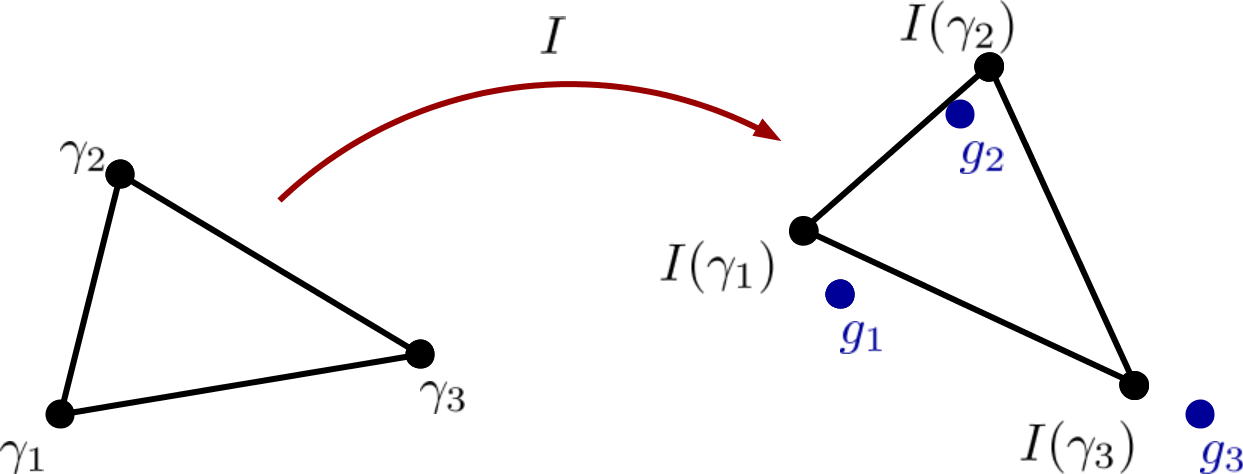}
	\caption{Existence of approximate configurations.}\label{f:conf}
\end{figure}

 It was observed by Bj\"orklund, Einsiedler, and Gorodnik \cite{beg}
 that the estimates on higher-order correlations (Theorem \ref{th:cor_high}) can be used to solve this problem in an optimal way when 
 $G$ is a connected simple Lie group with finite center,  and $\Gamma$ is a discrete subgroup of $G$ with finite covolume. 
 It is clear that since $\Gamma$ is discrete, the approximation
  \eqref{eq:conf} can not hold when the points $g_i$ are not too ``clustered''
  together. To address this issue, we introduce the notion of {\it width}:
  for $(g_1,\ldots,g_r)\in G^r$, we set
  $$
  {\sf w}(g_1,\ldots,g_r)=\min_{i\ne j} d(g_i,g_j). 
  $$
  We shall show that \eqref{eq:conf} can be established provided that
  the points are sufficiently spread out in terms of $\epsilon$.
  
  \begin{Theorem}\label{th:conf}
  For every $r\ge 2$, there exist $c_r,\epsilon_r>0$ such that for all tuples 
  $(g_1,\ldots,g_r)\in G^r$ satisfying
  \begin{equation*}
  {\sf w}(g_1,\ldots,g_r)\ge c_r\log(1/\epsilon)\quad \hbox{with $\epsilon\in (0,\epsilon_r)$},
  \end{equation*}
  there exists a tuples $(\gamma_1,\ldots,\gamma_r)\in \Gamma^r$ 
  and $g\in G$ such that
  $$
  d(g_i,g\cdot \gamma_i)<\epsilon\quad\hbox{for $i=1,\ldots,r$.}
  $$
  \end{Theorem}
  
 Let us illustrate Theorem \ref{th:conf} by an example of the orbit $\Gamma \cdot i$  in the hyperbolic plane $\mathbb{H}^2$ for $\Gamma=\hbox{PSL}_2(\mathbb{Z})$. 
 For $g\in \hbox{PSL}_2(\mathbb{R})$,  
  $$
  d(g i,i)=\cosh^{-1} (\|g\|^2/2)
  $$
 where $\|\cdot\|$ is the Euclidean norm. In this case, Theorem \ref{th:conf} with $r=2$
 reduces to showing that any distance $D>0$ can be approximated by distances from the set
  $$
 \Delta=\{\cosh^{-1} ((a^2+b^2+c^2+d^2)/2):\,\, a,b,c,d\in\Z^4,\, ad-b c=1  \}.
 $$
 Namely, when $D\ge c_2 \log(1/\epsilon),$
 there exists $\delta\in \Delta$ such that
 $|D-\delta|<\epsilon.$ 
 On the other hand, one can show that the set $\Delta$ is not $\epsilon$-dense in an
 interval $[a_\epsilon,\infty)$ with $a_\epsilon=o(\log(1/\epsilon))$ as $\epsilon\to 0^+$.

\begin{proof}[Proof of Theorem \ref{th:conf}]
We consider the action of $G$ on the space $X=G/\Gamma$ equipped
 with the normalised invariant measure $\mu$ and 	
apply Theorem \ref{th:cor_high} to a suitably chosen family of test 
functions supported on $X$. We take nonnegative $\tilde \phi_\epsilon\in C_c^\infty(G)$
such that 
$$
\hbox{supp}(\tilde \phi_\epsilon)\subset B_\epsilon(e),\quad
\mu(\tilde \phi_\epsilon)=1,\quad
S_\ell (\tilde \phi_\epsilon) \ll \epsilon^{-\alpha},
$$
for some fixed $\alpha>0$ depending only on $\ell$ and $G$.
Such a family of function can be constructed using
a local coordinate system in a neighbourhood of identity in $G$. We set
$$
\phi_\epsilon (g\Gamma)=\sum_{\gamma\in \Gamma} \tilde \phi_\epsilon(g\gamma),\quad g\in G,
$$
which defines a function in $C_c^\infty(G/\Gamma)$.
Then Theorem \ref{th:cor_high} give that
\begin{align*}
\mu((g_1\cdot \phi_\epsilon)\cdots (g_r\cdot \phi_\epsilon))=1
+ O_r \left( e^{-\delta\, {\sf w}(g_1,\ldots,g_r)} \epsilon^{-\alpha r}\right)
=1
+ O_r \left( \epsilon^{c_r\delta -\alpha r }\right).
\end{align*}
If we take $c_r>\alpha r/\delta$, then it follows from this estimate that
for all sufficiently small $\epsilon$,
$$
\mu((g_1\cdot \phi_\epsilon)\cdots (g_r\cdot \phi_\epsilon))>0.
$$
Since 
$$
\mu((g_1\cdot \phi_\epsilon)\cdots (g_r\cdot \phi_\epsilon))=\int_{G/\Gamma} \left(\sum_{\gamma_1,\ldots,\gamma_r\in \Gamma} \tilde\phi_\epsilon (g_1^{-1}g\gamma_1)\cdots \tilde\phi_\epsilon (g_r^{-1}g\gamma_r)\right)\, d\mu(g\Gamma),
$$
it follows that there exist 
$(\gamma_1,\ldots,\gamma_r)\in \Gamma^r$ and $g\in G$ such that
$$
g_i^{-1}g\gamma_i\in \hbox{supp}(\tilde \phi_\epsilon)\subset B_\epsilon(e)\quad\hbox{for $i=1,\ldots,r$},
$$
so that
$$
d(g_i, g\gamma_i)=d(g_i^{-1}g\gamma_i,e)<\epsilon\quad\hbox{for $i=1,\ldots,r$},
$$
as required.
\end{proof}

\section{Application: Central Limit Theorem}\label{sec:clt}

Suppose that the time evolution of a physical system is given by a one-parameter flow $T_t:X\to X$ on the phase space $X$. Observables of this system are represented by functions $\phi$
on $X$ so that studying the transformation of this system 
as time progresses involves the analysis of the values
$\phi(T_tx)$ with $t\ge 0$ and $x\in X$. Often these values fluctuate quite erratically
which makes it difficult to understand them in deterministic terms.
Instead, one might attempt to study their statistical properties.
Formally, we consider $\{\phi\circ T_t:\, t \ge 0\}$ as a family of random 
variables on $X$. For chaotic flows, these family typically exhibits quasi-independence
properties, and it is natural to expect that they satisfy probabilistic limit laws
known for independent random variables.
For instance, we mention one of the first results in this direction
which was proved by Sinai \cite{S}:

\begin{Theorem}\label{th:sinai}
Let $g_t:T^1(M)\to T^1(M)$ be  the geodesic flow on 
on a compact manifold $M$ with constant negative curvature. Then for any $\phi\in C^{1+\alpha}(X)$
with zero integral, the family of functions
$$
F_t(x)=t^{-1/2}\int_0^t \phi(g_sx)\,ds
$$
converges in distribution to the Normal Law as $t\to\infty$; that is, for all $\xi\in \mathbb{R}$,
$$
\frac{\hbox{\rm vol}\big(\{x\in T^1(M): F_t(x)<\xi\}\big)}{\hbox{\rm vol}(T^1(M))}\longrightarrow \hbox{\rm Norm}_{\sigma(\phi)}(\xi)\quad\hbox{as $t\to\infty$, }
$$
where 
\begin{align*}
 \hbox{\rm Norm}_{\sigma}(\xi)&=
 (\sqrt{2\pi}\sigma)^{-1}\int_{-\infty}^\xi e^{-s^2/(2\sigma^2)}\, ds
\end{align*}
denotes the Normal Distribution with variance $\sigma$.
\end{Theorem}

Validity of the Central Limit Theorem for one-parameter dynamical systems
has been extensively studied in the last decades, and 
we refer to surveys \cite{D,Den,gou,lB2,v} for an introduction 
to this vast area of research. However, there was very little known 
about the distribution of averages for more general groups actions.
In this section, we present a method developed by Bj\"orklund and Gorodnik \cite{bg}
for proving the Central Limit Theorem, which is based on the quantitative estimates
for higher-order correlations established in the previous sections.

\medskip

Let us consider an action of a group $H$ on a standard probability space $(X,\mu)$.
Given a function $\phi$ on $X$, we consider the family of its translations 
$$
(h\cdot\phi)(x)=\phi(h^{-1}x)\quad\hbox{with $h\in H$.}
$$
One may think about $\{h\cdot\phi:\,h\in H\}$ as a collection of 
identically distributed random variables on the probability space $(X,\mu)$.
When the action exhibits chaotic behaviour, it is natural to expect
that these random variable are quasi-independent in a suitable sense
which leads to the question whether these random variables satisfy
analogues of the standard probabilistic laws such as, for instance, 
the Central Limit Theorem, the Law of Iterated Logarithms, etc.

Here we prove a general Central Limit theorem for group actions.
From the perspective of this notes, the chaotic
nature of group actions is reflected in the asymptotic behaviour
of the higher-order correlations. We demonstrate that
quantitative estimates on correlations imply the Central Limit Theorem.
Although we do not pursue this direction here, we mention that this approach
has found interesting applications in Number Theory
to study the distribution of arithmetic counting functions
(see \cite{bg1,bg2}). 

Let $H$ be a (noncompact) locally compact group $H$ equipped with a left-invariant metric $d$.
We consider a measure-preserving action of $H$ on a standard probability space $(X,\mu)$.
We assume that this action is mixing of all orders in the following 
quantitative sense. There exists a subalgebra $\mathcal{A}$ of $L^\infty(X)$
equipped with a family of norms 
$$
S_1\le S_2\le \cdots \le S_\ell\le \cdots
$$
satisfying the following properties:
 \begin{enumerate}
 	\item[(${\rm N}_1$)] there exists $\ell_1$ such that
 	$$
 	\|\phi\|_{L^\infty}\ll S_{\ell_1}(\phi),
 	$$
 	\item[(${\rm N}_3$)] for all $\ell$, there exists $\sigma_\ell>0$ such that
 	$$
 	S_\ell(g\cdot \phi)\ll_\ell e^{\sigma_\ell\, d(g,e)} \, S_\ell(\phi)\quad\hbox{for all $g\in G$,}
 	$$
 	\item[(${\rm N}_4$)] for every $\ell$, there exists $\ell'$ such that
 	$$
 	S_\ell(\phi_1\phi_2)\ll_\ell S_{\ell'}(\phi_1) S_{\ell'}(\phi_2).
 	$$
 \end{enumerate}
We suppose that for every $r\ge 2$
there exist $\delta_r,\ell_r>0$
such that for all elements $h_1,\ldots,h_r\in H$ and all functions $\phi_1,\ldots,\phi_r\in \mathcal{A}$,
\begin{align}\label{eq:mix_end}
\mu((h_1\cdot \phi_1)\cdots (h_r\cdot \phi_r))= \,&\mu(\phi_1)\cdots \mu(\phi_r)\\
&+
O_r\left( S_{\ell_r}(\phi_1)\cdots S_{\ell_r}(\phi_r)\, e^{-\delta_r D(h_1,\ldots,h_r)}\right), \nonumber
\end{align}
where 
$$
D(h_1,\ldots,h_r)=\min_{i\ne j} d(h_i,h_j).
$$

We shall additionally assume that  the group $H$ has {\it subexponential growth}
which means tha the balls
$$
B_t=\{h\in H:\, d(h,e)<t\}
$$
satisfy
\begin{equation}
\label{eq:subexp}
\frac{\log\hbox{vol}(B_t)}{t}\to 0\quad\hbox{as $t\to\infty$.}
\end{equation}

Our main result is the following:

\begin{Theorem}\label{th:clt}
For every $\phi\in \mathcal{A}$ with integral zero,
the family of functions 
\begin{equation}
\label{eq:f_t}
F_t(x)=\hbox{\rm vol}(B_t)^{-1/2} \int_{B_t} \phi(h^{-1}x)\,dh
\end{equation}
converges in distribution as $t\to \infty$ to the Normal Law with mean zero and the variance 
$$
\sigma(\phi)^2=\int_H \left<h\cdot \phi,\phi\right>dh.
$$
Explicitly, this means that for every $\xi\in \mathbb{R}$,
$$
\mu\big(\left\{x\in X:\, F_t(x)<\xi \right\}\big)\longrightarrow \hbox{\rm Norm}_{\sigma(\phi)}(\xi)\quad\hbox{as $t\to \infty$.}
$$
\end{Theorem}

We remark that the condition that $H$ has subexponential growth is important in Theorem \ref{th:clt}.
Indeed, Gorodnik and Ramirez \cite{gor_r} constructed examples of actions of rank-one simple Lie groups
on homogeneous spaces which are exponentially mixing of all orders, but do not satisfy
the Central Limit Theoorem.

In particular, Theorem \ref{th:clt} immediately implies the following 
results about higher-rank abelian actions on homogeneous spaces.

\begin{Cor}\label{cor:cartan}
Let $G$ be a (noncompact) connected simple matrix Lie group with finite centre
and $H$ a (noncompact) closed subgroup of a Cartan subgroup of $G$. 
Then a measure-preserving action of $H$ on finite-volume homogeneous spaces $X$ of $G$
satisfies the Central Limit Theorem. Namely, for every $\phi\in C_c^\infty(X)$
with zero integral, the family of functions
$$
F_t=\hbox{\rm vol}(B_t)^{-1/2} \int_{B_t} (h\cdot \phi)\,dh
$$
converges in distribution to the Normal Law as $t\to\infty$.
\end{Cor}

We note that when $M$ is compact surface with constant negative curvature,
its unit tangent bundle can be realised as
$$
T^1(M)\simeq \hbox{PSL}_2(\R)/\Gamma,
$$
where $\Gamma$ is a discrete cocompact subgroup of $\hbox{PSL}_2(\R)$, and 
the geodesic flow is given by
$$
g_t:x\mapsto 
\left(
\begin{tabular}{cc}
$e^{t/2}$ & 0 \\
0 & $e^{-t/2}$
\end{tabular}
\right)
x\quad \hbox{for $x\in \hbox{PSL}_2(\R)/\Gamma.$}
$$
Hence, our method also provides a new proof of Theorem \ref{th:sinai}.

\medskip

It is well-known from Probability that in order to establish that
a family of bounded random variables $X_t$ converges in distribution
to a normal random variable $N$, it is sufficient to establish convergence of 
all moments, that is, that for all $r\ge 1$
$$
\mathbb{E}(X_t^r)\to \mathbb{E}(N^r) \quad\hbox{as $t\to\infty$.}
$$
We essentially follows this route, but it will be more convenient 
to work with cumulants instead of moments.
Given random variables $X_1,\ldots, X_r$,
the joint {\it cumulant} is defined as
$$
\cum(X_1,\ldots,X_r)=(-i)^r\frac{\partial^r}{\partial z_1\cdots \partial z_r}
\log \mathbb{E} \left[ e^{i\sum_{k=1}^r z_k X_k} \right] \Big|_{z_1=\cdots=z_r=0}.
$$
It is useful to keep in mind that the joint cumulants can be expressed 
in terms of joint moments and conversely (see, for instance, \cite{ls}):
\begin{align*}
\cum(X_1,\ldots,X_r)&=
\sum_{P\in\mathcal{P}_r} (-1)^{|P|-1} (|P|-1)!\,{\prod}_{I\in P} \mathbb{E}\left({\prod}_{i\in I }X_i\right),\\
\mathbb{E}(X_1\cdots X_r)&=\sum_{P\in \mathcal{P}_r } {\prod}_{I\in P} \cum(X_i:\, i\in I),
\end{align*}
where the sums are taken over the set $\mathcal{P}_r$ consisting of all partitions of $\{1,\ldots,r\}$.
Hence, studying cumulants is essentially equivalent to studying moments.
However, it turns our that cumulants have several very convenient 
additional vanishing properties that will be crucial for our argument:
\begin{itemize}
\item If there exists a nontrivial partition $\{1,\ldots,r\}=I\sqcup J$
such that $\{X_i:i\in I\}$ and $\{X_i:i\in J\}$ are independent, then
\begin{equation}
\label{eq:indep}
\cum(X_1,\ldots,X_r)=0.
\end{equation}
 \item if $N$ is a normal random variable, than the cumulants of order at least three satisfy
 $$
 \cum(N,\ldots,N)=0
 $$
\end{itemize}

\medskip

Now we adopt this probabilistic notation to our setting.
For functions $\phi_1,\ldots,\phi_r\in L^\infty(X)$ and a subset $I\subset \{1,\ldots, r\}$,
we set 
$$
\phi_I={\prod}_{i\in I} \phi_i.
$$
We use the convention that $\phi_\emptyset=1$.
Then we define the joint {\it cumulant} of $\phi_1,\ldots,\phi_r$ as
$$
\cum_r(\phi_1,\ldots,\phi_r)=\sum_{P\in\mathcal{P}_r} (-1)^{|P|-1} (|P|-1)!\,{\prod}_{I\in P} \mu(\phi_I).
$$
For a function $\phi\in L^\infty(X)$, we also set
$$
\cum_r(\phi)=\cum_r(\phi,\ldots,\phi).
$$

The following proposition,
which is essentially equivalent to the more widely known Method of Moments,
provides a convenient criterion for proving the Central Limit Theorem.

\begin{Prop}\label{p:cum}
Let $F_t\in L^\infty(X)$ be a family of functions such that as $t\to\infty$,
\begin{align}
\mu(F_t)&\to 0,	\label{eq:c1}\\
 \|F_t\|_{L^2}&\to \sigma, \label{eq:c2} \\
 \cum_r(F_t)&\to 0\quad\hbox{for all $r\ge 3$.} \label{eq:c3}
\end{align}
Then
for every $\xi\in \mathbb{R}$,
$$
\mu\big(\left\{x\in X:\, F_t(x)<\xi \right\}\big)\longrightarrow \hbox{\rm Norm}_{\sigma}(\xi)
\quad\hbox{as $t\to \infty$.}
$$
\end{Prop}

Estimates on cumulants were also used by Cohen and Conze \cite{CC1,CC2,CC3} to prove the Central Limit Theorem for $\Z^k$-actions by automorphisms of compact abelian groups. 

\medskip

We begin the proof of Theorem \ref{th:clt}.
In view of Proposition \ref{p:cum}, it remains to verify that the family of functions $F_t$ defined in \eqref{eq:f_t}
satisfies \eqref{eq:c1}, \eqref{eq:c2}, and \eqref{eq:c3}. The first condition
is immediate, and the second is verified as follows. 
We observe that 
\begin{align*}
\|F_t\|_{L^2}^2 &= \hbox{vol}(B_t)^{-1}\int_{B_t\times B_t} \left<h_1\cdot \phi,h_2\cdot \phi\right>\, dh_1dh_2\\
&=\hbox{vol}(B_t)^{-1}\int_{H\times H} \chi_{B_t}(h_1)\chi_{B_t}(h_2)\left<(h_1^{-1}h_2)\cdot \phi, \phi\right>\, dh_1dh_2\\
&=\int_{H} \frac{\hbox{vol}(B_t\cap B_t h^{-1})}{\vol(B_t)}
\left<h\cdot \phi, \phi\right>\, dh.
\end{align*}
It is not hard to check using the subexponential growth property \eqref{eq:subexp} that 
the balls $B_t$ satisfy the F\o lner property, that is, for all $h\in H$,
$$
\frac{\hbox{vol}(B_t\cap B_t h^{-1})}{\vol(B_t)}\to 1\quad\hbox{as $t\to\infty$}.
$$
Moreover, it follows from \eqref{eq:mix_end} with $r=2$ that
the function $h\mapsto \left<h\cdot \phi, \phi\right>$ is in $L^1(H)$.
Thus, using the Dominated Convergence Theorem, we deduce that
$$
\|F_t\|_{L^2}^2\to \int_H \left<h\cdot \phi, \phi\right>\, dh\quad\hbox{as $t\to\infty$.}
$$
This implies \eqref{eq:c2}.

\medskip

Verification of \eqref{eq:c3} is the most challenging part of the proof
because  it requires to show asymptotic vanishing of the cumulants
$$
\cum_r(F_t)=\hbox{vol}(B_t)^{-r/2}\int_{B_t^r} \cum_r(h_1\cdot \phi,\ldots,h_r\cdot \phi)\, d{h},
$$
which is even more than the square-root cancellation in this integral. 
The first crucial input for estimating $\cum_r(F_t)$ comes from the bound
on correlations \eqref{eq:mix_end}. However, these bound will be only useful
for certain ranges of tuples $h=(h_1,\ldots,h_r)$. 

To utilise the bound \eqref{eq:mix_end} most efficiently,
we introduce a decomposition of the product $H^r$ into a union
of domains where the components $h_i$ are either separated or clustered
on suitable scales. For subsets $I,J\subset \{1,\ldots,r\}$ and $h=(h_1,\ldots,h_r)\in H^r$,
we set
\begin{align*}
d^I(h)&=\max\{d(h_i,h_j):\, i,j\in I\}\\
d_{I,J}(h)&=\min\{d(h_i,h_j):\, i\in I, j\in J\},
\end{align*}
and for a partition $Q\in\mathcal{P}_r$, we set
\begin{align*}
d^Q(h)&=\max\{ d^I(h):\, I\in Q\},\\
d_Q(h)&=\min\{d_{I,J}(h):\, I\ne J\in Q\}.
\end{align*}
Using this notation, we define for $0\le \alpha\le\beta$,
\begin{align*}
\Delta_Q(\alpha,\beta)&=\{h\in H^r:\, d^Q(h)\le \alpha,\ d_Q(h)>\beta \},\\
\Delta(\beta)&=\{h\in H^r:\, d(h_i,h_j)\le \beta \quad\hbox{for all $i,j$}\}.
\end{align*}
For $h=(h_1,\ldots,h_r)\in \Delta_Q(\alpha,\beta)$, we think about components
$h_i$ with $i$ in the same atom of $Q$ as ``clustered'' and about 
$h_i$ with $i$ in different atoms of $Q$ as ``separated'' (see Figure \ref{f:tuples}).
\begin{figure}[h]
	\includegraphics[width=0.7\linewidth]{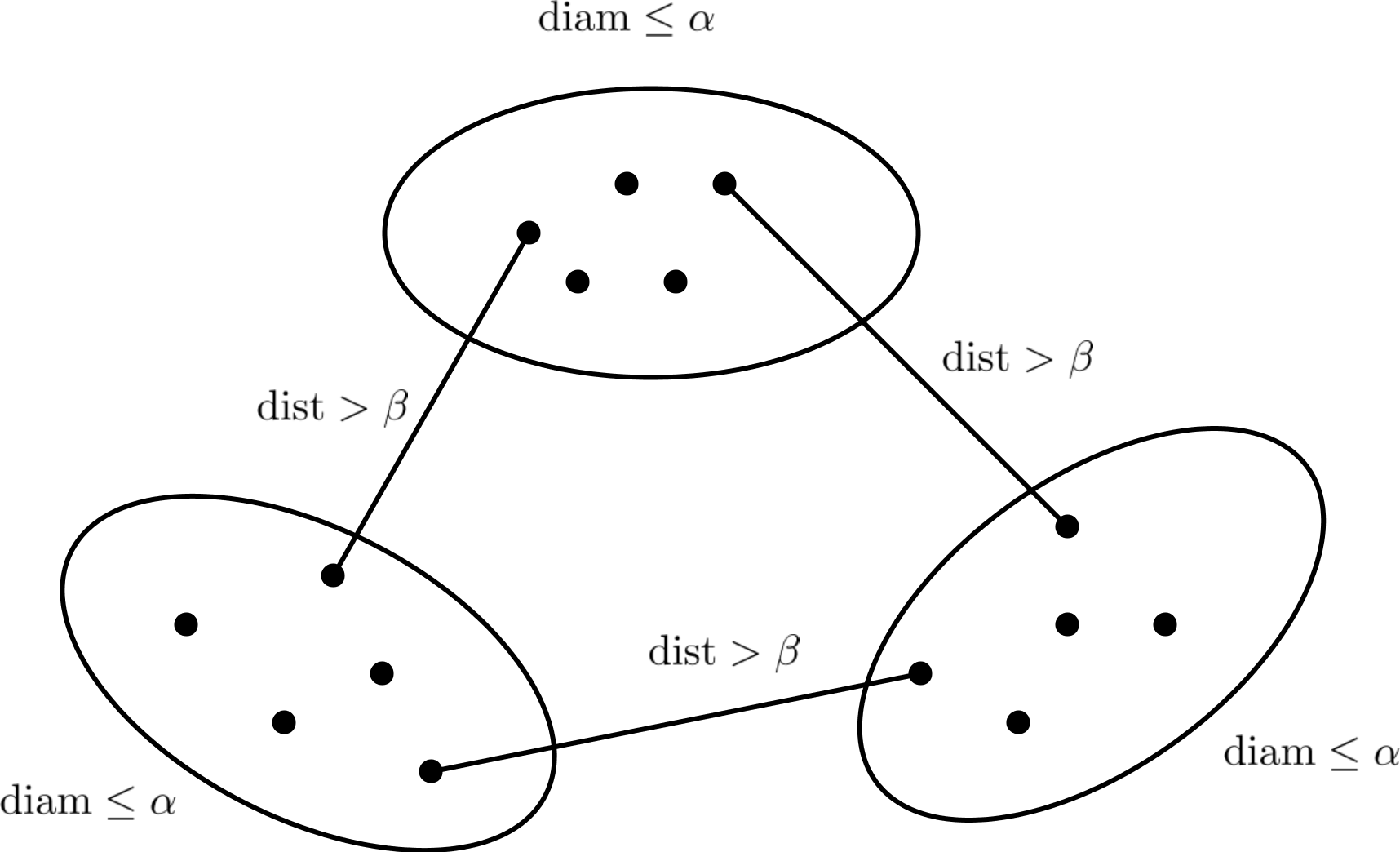}
	\caption{Tuples in the sets $\Delta(\alpha,\beta)$.}\label{f:tuples}
\end{figure}

These features allow to estimate $\cum_r(h_1\cdot \phi,\ldots,h_r\cdot \phi)$
on the sets $\Delta_Q(\alpha,\beta)$:

\begin{Prop}\label{p:bound}
There exist $\delta_r,\sigma_r>0$ such that for every 
$0\le\alpha\le \beta$, $Q\in\mathcal{P}_r$ with $|Q|\ge 2$, and 
$(h_1,\ldots,h_r)\in \Delta_Q(\alpha,\beta)$,
$$
\cum_r(h_1\cdot \phi,\ldots,h_r\cdot \phi)\ll_{r,\phi} e^{-\delta_r \beta-\sigma_r\alpha}.
$$
\end{Prop}

\begin{proof}
The proof will exploit a certain cancellation property of cumulants.
For $Q\in \mathcal{P}_r$ and $\phi_1,\ldots,\phi_r\in L^\infty(X)$, 
we define the conditional cumulant as 
$$
\cum_r(\phi_1,\ldots,\phi_r|Q)=\sum_{P\in\mathcal{P}_r} (-1)^{|P|-1} (|P|-1)!\prod_{I\in P}
\prod_{J\in Q} \mu(\phi_{I\cap J}).
$$
 One can show that 
when the partition $Q$ is nontrivial,
\begin{equation}
\label{eq:cum_cond}
\cum_r(\phi_1,\ldots,\phi_r|Q)=0.
\end{equation}
This fact is an analogue of \eqref{eq:indep}, but
it is not a probabilistic property, but rather a combinatorial
cancellation feature of the cumulant sums, and we refer, for instance, to \cite{bg}
for a self-contained proof of \eqref{eq:cum_cond}.

In order to bound $\cum_r(h_1\cdot \phi,\ldots,h_r\cdot \phi)$,
we shall show that when $(h_1,\ldots,h_r)\in \Delta_Q(\alpha,\beta)$,
$$
\cum_r(h_1\cdot \phi,\ldots,h_r\cdot \phi)\approx \cum_r(h_1\cdot \phi,\ldots,h_r\cdot \phi|Q)
$$
which reduces to verifying that for $I\in P$,
$$
\mu\left({\prod}_{i\in I} h_i\cdot \phi \right)\approx 
{\prod}_{J\in Q} \mu\left({\prod}_{i\in I\cap J} h_i\cdot \phi \right).
$$
This is where the full strength of the estimate \eqref{eq:mix_end} on higher-order correlations comes into play.
For each $J$, we pick $h_J$ as one of $h_j$, $j\in I\cap J$. Then
$$
\mu\left({\prod}_{i\in I} h_i\cdot \phi_i \right)=
\mu\left({\prod}_{J\in  Q} h_{J} \Phi_J \right),
$$
where $\Phi_J=\prod_{i\in I\cap J} (h_J^{-1}h_i)\cdot \phi$.
Since $(h_1,\ldots,h_r)\in \Delta_Q(\alpha,\beta)$, we have
\begin{align*}
&d(h_J^{-1}h_i,e)=d(h_i,h_J)\le \alpha  & \hbox{for $i\in J\in Q$},\\
&d(h_{J_1},h_{J_2})>\beta & \hbox{for $J_1\ne J_2\in Q$.}
\end{align*}
Hence, it follows from \eqref{eq:mix_end} that 
$$
\mu\left({\prod}_{J\in  Q} h_{J} \Phi_J \right) 
={\prod}_{J\in Q} \mu(\Phi_J)+O_r\left({\prod}_{J\in Q} S_{\ell_r}(\Phi_J) e^{-\delta_r\beta}\right),
$$
and by the properties (${\rm N}_3$) and (${\rm N}_4$) of the norms,
$$
S_{\ell}(\Phi_J)\ll_\ell {\prod}_{i\in I\cap J} S_{\ell'}((h_J^{-1}h_i)\cdot \phi)
\ll_{\ell',\phi} e^{r\sigma_\ell \alpha}.
$$
This implies that for some $\sigma_r>0$,
$$
\mu\left({\prod}_{i\in I} h_i\cdot \phi_i \right)=
{\prod}_{J\in Q} \mu\left({\prod}_{i\in I\cap J} h_i\cdot \phi_i \right)
+O_{r,\phi}\left( e^{-(\delta_r \beta-\sigma_r\alpha)} \right).
$$
which can be used to prove the proposition.
\end{proof}

We shall use the following decomposition of the space of tuples $H^r$.

\begin{Prop}\label{prop_decom}
	Given parameters 
	$$
	0=\beta_0<\beta_1< 3\beta_1\le\beta_2<\cdots <\beta_{r-1}<3\beta_{r-1}\le\beta_r,
	$$
	we  have the decomposition
	$$
	H^r=\Delta(\beta_r)\cup \left( \bigcup_{j=0}^{r-1} \bigcup_{Q:\,|Q|\ge 2} \Delta_Q(3\beta_j,\beta_{j+1}) \right).
	$$
\end{Prop}

The proof of Proposition \ref{prop_decom} uses the following lemma:

\begin{Lemma}\label{l:coarse}
Let $Q\in \mathcal{P}_r$ with $|Q|\ge 2$ and $0\le\alpha\le \beta$.
Suppose that for $h\in H^r$,
$$
d^Q(h)\le \alpha\quad\hbox{and}\quad d_Q(h)\le \beta.
$$
Then there exists a partition $Q_1$ which is strictly coarser than $Q$ such that 
$$
d^{Q_1}(h)\le 3\beta.
$$
\end{Lemma}

\begin{proof}
We observe that the sets $\{h_i:\, i\in I\}$ with $I\in Q$ have diameters at most $\alpha$,
and the distance between at least two of these sets is bounded by $\beta$.
We define the new partition $Q_1$ by combining the sets whose distance at most $\beta$
between them. This gives a strictly coarser partition.
It follows from the triangle inequality that 
the diameters of the sets $\{h_i:\, i\in J\}$ with $J\in Q_1$
are at most $2\alpha+\beta\le 3\beta$. This implies that $d^{Q_1}(h)\le 3\beta$.
\end{proof}

\begin{proof}[Proof of Proposition \ref{prop_decom}]
Let us take arbitrary $h\in H^r$. Suppose that $h\not\in \Delta_{Q_0}(0,\beta_1)$ for 
$Q_0=\{\{1\},\ldots,\{r\}\}$. It is clear that $d^{Q_0}(h)=0$ so that also 
$d_{Q_0}(h)\le \beta_1$. Hence, it follows from Lemma \ref{l:coarse}
that there exists a partition $Q_1$ coarser than $Q_0$ such that 
$d^{Q_1}(h)\le 3\beta_1$. If $d_{Q_1}(h)>\beta_2$, then $h\in \Delta_{Q_1}(3\beta_1,\beta_2)$
and $h$ belongs to the union. On the other, if $d_{Q_1}(h)\le \beta_2$,
we apply Lemma \ref{l:coarse} again to conclude that there exists a partition $Q_2$
coarser than $Q_1$ such that $d^{Q_2}(h)\le 3\beta_2$. 
This argument can be continued, and we deduce that after at most $r$ steps,
we see that $h$ belongs to the union of $\Delta_{Q_j}(3\beta_j,\beta_{j+1})$ with $|Q_j|\ge 2$,
or we get $Q_i=\{\{1,\ldots,r\}\}$ and  $d^{Q_i}(h)\le 3\beta_i<\beta_r$.
In the latter case, we deduce that $h\in\Delta(\beta_r)$. This proves
the required decomposition. 
\end{proof}	

Now we are ready to complete the proof of Theorem \ref{th:clt}.
As we have already remarked, it remains to prove \eqref{eq:c3}.
Using the decomposition established in Proposition \ref{prop_decom}, we deduce that
\begin{align*}
\cum_r(F_t)=&\,\hbox{vol}(B_t)^{-r/2}\int_{B_t^r} \cum_r(h_1\cdot \phi,\ldots,h_r\cdot \phi)\, d{h}\\
\ll_r &\, \hbox{vol}(B_t)^{-r/2}\Big(\hbox{vol}(B_t^r\cap \Delta(\beta_r))\|\phi\|^r_{L^\infty} \\
&\quad\quad\quad\quad\quad\;+\max_{j,\, Q:|Q|\ge 2}
\int_{B_t^r\cap\Delta_Q(3\beta_j,\beta_{j+1})} |\cum_r(h_1\cdot \phi,\ldots,h_r\cdot \phi)|\, dh \Big).
\end{align*}
It follows from invariance of the volume on $H$ that
$$
\hbox{vol}(B_t^r\cap \Delta(\beta_r))
\le  \int_{B_t} \vol(B(h,\beta_r))^{r-1}\, dh=\hbox{vol}(B_t)\vol(B_{\beta_r})^{r-1}.
$$
Hence, using Proposition \ref{p:bound}, we conclude that
$$
\cum_r(F_t)\ll_{r,\phi} \hbox{vol}(B_t)^{1-r/2}\vol(B_{\beta_r})^{r-1}+\hbox{vol}(B_t)^{r/2}\Big(\max_j e^{-\delta_r\beta_{j+1}-3\sigma_r\beta_j}\Big).
$$
For a parameter $\theta>0$, we choose $\beta_j$'s recursively as 
$$
\beta_0=0,\, \beta_{j+1}=\max\{3\beta_j, \delta_r^{-1}(\theta +3\sigma_r \beta_j) \}.
$$
Then $\beta_r\le c_r\theta$ with some $c_r>0$, and 
$$
\cum_r(F_t)\ll_{r,\phi} \hbox{vol}(B_t)^{1-r/2}\vol(B_{c_r\theta})^{r-1}+\hbox{vol}(B_t)^{r/2} e^{-\theta}.
$$
We take $\theta=r\log\hbox{vol}(B_t)$. Then it follows from 
the subexponential growth condition \eqref{eq:subexp} that when $r\ge 3$,
$$
\cum_r(F_t)\to 0\quad \hbox{as $t\to\infty$.}
$$
This completes the proof of Theorem \ref{th:clt}.

\end{document}